\documentclass{amsart}
\pdfoutput=1
\newcommand{\sv}[1]{}
 \newcommand{\lv}[1]{#1}
\usepackage{etoolbox}

 \newcommand{\toappendix}[1]{#1}
\usepackage[utf8]{inputenc}
\usepackage[T1]{fontenc}
\lv{\usepackage{lmodern}}
\lv{\usepackage[foot]{amsaddr}}
\usepackage{microtype}
\usepackage{amsfonts}
\usepackage{graphicx}
\usepackage{tabularx}
\usepackage{array}
\lv{\usepackage{amsmath}}
\usepackage{amsthm}
\usepackage{amssymb}
\lv{\usepackage{fullpage}}
\usepackage[dvipsnames]{xcolor}
\usepackage{tikz}
\lv{\usepackage{hyperref}}

\tikzstyle{legend_general}=[rectangle, rounded corners, thin,
                          top color= white,bottom color=lavander!25,
                          minimum width=2.5cm, minimum height=0.8cm,
                          violet]

\DeclareMathOperator{\col}{col}

\DeclareMathOperator{\ch}{ch}
\lv{
\newtheorem{theorem}{Theorem}[section]
\newtheorem{proposition}[theorem]{Proposition}

\newtheorem{claim}[theorem]{Claim}
\newtheorem{lemma}[theorem]{Lemma}
\newtheorem{corollary}[theorem]{Corollary}
\theoremstyle{definition}
\newtheorem{definition}[theorem]{Definition}

}
\newtheorem{observation}[theorem]{Observation}
\newtheorem{question}[theorem]{Question}

\newcommand{\td}{\ensuremath{\sf{td}}}

\renewcommand{\epsilon}{\varepsilon}
\renewcommand{\phi}{\varphi}
\newcommand{\brm}[1]{\operatorname{#1}}

\title{Flexible List Colorings in Graphs with Special Degeneracy Conditions}\lv{\thanks{T.~Masařík did this work while he was a postdoc at Simon Fraser University in Canada. L.~Stacho was supported by NSERC grant R611368.\newline
    An extended abstract of this manuscript has appeared at the 31st International Symposium on Algorithms and Computation (ISAAC) 2020~\cite{ISAAC}. \newline
Data sharing is not applicable to this article as no new data were created or analyzed in this study.
}} 
\lv{
\author{Peter Bradshaw}
\author{Tomáš Masařík}
\author{Ladislav Stacho}
\address{Simon Fraser University, Burnaby, BC, Canada}
\email{pabradsh@sfu.ca, masarik@kam.mff.cuni.cz, ladislav\_stacho@sfu.ca}
}

\begin{document}
\maketitle
\begin{abstract}
  For a given $\varepsilon > 0$, we say that a graph $G$ is \emph{$\epsilon$-flexibly $k$-choosable} if the following holds: for any assignment $L$ of color lists of size $k$ on $V(G)$, if a preferred color from a list is requested at any set $R$ of vertices, then at least $\epsilon |R|$ of these requests are satisfied by some $L$-coloring.
We consider 
the question of flexible choosability in several graph classes with certain degeneracy conditions.
We characterize the graphs of maximum degree $\Delta$ that are $\epsilon$-flexibly $\Delta$-choosable for some $\epsilon = \epsilon(\Delta) > 0$, which answers a question of Dvo\v{r}\'ak, Norin, and Postle [List coloring with requests, JGT 2019].
In particular, we show that for any $\Delta\geq 3$, any graph of maximum degree $\Delta$ that is not isomorphic to $K_{\Delta+1}$ is $\frac{1}{6\Delta}$-flexibly $\Delta$-choosable.
Our fraction of $\frac{1}{6 \Delta}$ is within a constant factor of being the best possible.
We also show that graphs of treewidth $2$ are $\frac{1}{3}$-flexibly $3$-choosable, answering a question of Choi et al.~[arXiv 2020], and we give conditions for list assignments by which graphs of treewidth $k$ are $\frac{1}{k+1}$-flexibly $(k+1)$-choosable. We show furthermore that graphs of treedepth $k$ are $\frac{1}{k}$-flexibly $k$-choosable. Finally, we introduce a notion of \emph{flexible degeneracy}, which strengthens flexible choosability, and we show that apart from a well-understood class of exceptions, 3-connected non-regular graphs of maximum degree $\Delta$ are flexibly $(\Delta - 1)$-degenerate.
\end{abstract}

\section{Introduction}
A \emph{proper coloring} of a graph $G$ is a function $\phi:V(G) \rightarrow S$ by which each vertex of $G$ receives a color from some color set $S$, such that no pair of adjacent vertices is assigned the same color.
Proper graph coloring is one of the oldest concepts in graph theory. 
The \emph{precoloring extension} problem is a related question which asks whether a graph can be properly colored using a given color set even when some vertices have preassigned colors. Surprisingly, the precoloring extension problem often has a negative answer, even for relatively simple graph classes and for a small number of precolored vertices~\cite{Tuza3,Marx06}.
In particular, it is NP-complete to decide whether an interval graph can be properly colored when only two vertices are precolored by different colors~\cite{Tuza1}.

In practical applications, proper graph coloring is often used to represent scheduling problems, in which case preassigned colors may be used to represent scheduling preferences or requests. Thus, the precoloring extension problem is not only an interesting concept in theory, but also has various practical applications, such as register allocation, scheduling, and many others (see~\cite{Tuza1} for an overview of basic precoloring extension applications).
However, in many applications that arise from graph coloring with requests, it is not always necessary to satisfy all coloring requests. With this in mind, Dvořák, Norin, and Postle~\cite{Dvorak} recently introduced a relaxed notion of the precoloring extension problem in which it is not mandatory to satisfy every coloring request and it is sufficient to satisfy a positive fraction of all coloring requests. They named this new concept \emph{flexibility}.

Dvořák, Norin, and Postle observed that for $k$-colorable graphs, the problem of finding a proper $k$-coloring that satisfies a positive fraction of some set of coloring requests is trivial, since by permuting the $k$ colors of any proper $k$-coloring, one can satisfy at least a fraction of $\frac{1}{k}$ of any set of coloring requests.
However, this approach does not work for list colorings, and thus the flexibility problem applied to list colorings becomes attractive and challenging.
A graph where each vertex $v$ has a list $L(v)$ of available colors is called \emph{$L$-colorable} if there exists a proper coloring in which each vertex $v$ receives a color from $L(v)$. We call such a coloring an \emph{$L$-coloring}.
A graph is \emph{$k$-choosable} if every assignment $L$ of at least $k$ colors to each vertex guarantees an $L$-coloring.
A graph class $\mathcal G$ is \emph{$k$-choosable} if every $G\in \mathcal G$ is $k$-choosable.
The \emph{choosability} of a graph $G$, written $\ch(G)$, is the minimum $k$ such that $G$ is $k$-choosable.

We now formally introduce the concept of flexibility.
A \emph{weighted request} on a graph $G$ with list assignment $L$ is a function $w$ that assigns a non-negative real number to each pair $(v,c)$ where $v\in V(G)$ and $c\in L(v)$.  
For $\varepsilon>0$, we say that $w$ is \emph{$\varepsilon$-satisfiable} if there exists an $L$-coloring $\varphi$ of $G$ such that
\[%
 \sum_{v\in V(G)} w(v,\varphi(v))\ge\varepsilon\cdot \sum\limits_{\substack{v\in V(G) \\ c\in L(v)}}
 w(v,c).
\]%
Then, we say that $G$ is \emph{weighted $\epsilon$-flexible} with respect to $L$ if every weighted request on $G$ is $\epsilon$-satisfiable.

We also consider an unweighted version of flexibility. We say that a \emph{request} on a graph $G$ with list assignment $L$ is a function $r$ with $\brm{dom}(r)\subseteq V(G)$ such that $r(v)\in L(v)$ for all $v\in\brm{dom}(r)$.\footnote{Note that if for all $v\in V(G)$, $w(v, c)\in\{0, 1\}$ holds for all $c\in L(v)$ and $w(v,c) = 1$ for at most one $c \in L(v)$, then a weighted request $w$ is actually a request.}
Analogously, for $\varepsilon>0$, we say that a request $r$ is \emph{$\varepsilon$-satisfiable} if there exists an $L$-coloring $\varphi$ of $G$ such that at least $\epsilon|\brm{dom}(r)|$ vertices $v$ in $\brm{dom}(r)$ receive the color $r(v)$. Then, we say that $G$ is $\epsilon$-flexible with respect to $L$ if every request on $G$ is $\epsilon$-satisfiable. In both the weighted and unweighted case, we call $\varepsilon$ the \emph{flexibility proportion} of $G$ with respect to $L$.

An interesting special case of unweighted flexibility, which was brought up recently in~\cite{choiclemen}, arises when each vertex of $G$ 
has exactly one preferred color, i.e.,
$\brm{dom}(r)=V(G)$.
We call such a request $r$ \emph{widespread}.
Analagously, we say that a graph $G$ with list assignment $L$ is \emph{weakly $\varepsilon$-flexible} if every widespread request is $\varepsilon$-satisfiable. 

If $G$ is $\varepsilon$-flexible for every list assignment with lists of size $k$, we say that $G$ is {\it $\varepsilon$-flexible for lists of size $k$.}
To simplify our terminology, we also often say that $G$ is {\it $\varepsilon$-flexibly $k$-choosable}. 
Analogously, if $\mathcal L$ is a set of assignments of color lists to $V(G)$, then we say that $G$ is {\it $\varepsilon$-flexibly $\mathcal L$-choosable} if $G$ is $\varepsilon$-flexible for every list assignment $L\in\mathcal L$.
For a graph class $\mathcal G$, we may omit $\varepsilon$ and say that $\mathcal G$ is {\it flexibly $k$-choosable} if there exists a universal constant $\varepsilon>0$ such that every graph in $\mathcal G$ is $\varepsilon$-flexibly $k$-choosable.
We will use this notation whenever we do not care about the precise value of the constant $\epsilon$ and only care that a constant $\epsilon > 0$ exists.
A meta-question which is central in the study of flexibility asks whether a given graph class is flexibly $k$-choosable.
We sometimes refer to this question simply as \textit{flexibility}.
Note that all the notation mentioned in this paragraph can be also stated for weak or weighted flexibility.

Typically, research in flexibility focuses on bounding the list size $k$ needed for a graph class to be $\epsilon$-flexibly $k$-choosable for some $\epsilon > 0$, and the precise value of $\epsilon$ is not usually of concern.
Apart from the original paper introducing flexibility~\cite{Dvorak}, where some basic results in terms of maximum average degree were established, the main focus in flexibility research has been on planar graphs. 
In particular, for many subclasses $\mathcal G$ of planar graphs, there has been a vast effort to reduce the gap between the choosability of $\mathcal G$ and the list size needed for flexibility in $\mathcal G$.
As of now, some tight bounds on list sizes are known:
namely, triangle-free\footnote{A graph $G$ is $\mathcal F$-free if $G$ does not contain any graph $F\in \mathcal F$ as a subgraph.} planar graphs~\cite{DvorakTriangle}, $\{C_4,C_5\}$-free planar graphs~\cite{Yang20}, and $\{K_4, C_5,C_6,C_7,B_5\}$-free\footnote{$B_5$ denotes the \emph{book} on $5$ vertices, which is the graph consisting of $3$ triangles sharing a common edge.} planar graphs~\cite{lidicky} are flexibly $4$-choosable, and planar graphs of girth 6~\cite{DvorakGirth6} are flexibly $3$-choosable.
For other subclasses $\mathcal G$ of planar graphs, an upper bound is known for the list size $k$ required for $\mathcal G$ to be flexibly $k$-choosable~\cite{choiclemen,Masarik}. 
However, these upper bounds are not known to be tight; see~\cite{choiclemen} for a comprehensive overview and a discussion of the related results.
The main question in this direction, of determining whether planar graphs are flexibly $5$-choosable, remains open.

\subsection{Our Results}
As discussed, the list size $k$ needed for a graph to be $\epsilon$-flexibly $k$-choosable for some $\epsilon > 0$ has a basic lower bound equal to the graph's choosability (the list size needed just for a proper list-coloring).
Similarly to list coloring, a graph's \emph{degeneracy} $d$, which is the largest minimum degree over all induced subgraphs, plays a natural role in establishing upper bounds on the list size needed for flexibility in a graph. 
However, while a simple greedy argument gives an upper bound of $d+1$ on the choosability, only the weaker upper bound of $d+2$ is currently known to hold for the list size needed for flexibility, as shown in~\cite{Dvorak}.
In the same paper, the authors ask whether an upper bound of $d+1$ can always be achieved---that is, whether $d$-degenerate graphs are flexibly $(d+1)$-choosable. However, answering this question seems to be out of reach with current knowledge.
Even for 2-degenerate graphs, the question seems rather tough, as it would imply the non-trivial result that planar graphs of girth $6$ are flexibly $3$-choosable, proven in \cite{DvorakGirth6}, as this class of graphs is $2$-degenerate.

In this direction, Dvo\v{r}\'ak et al.~\cite{Dvorak} asked a more specific question about non-regular\footnote{A non-regular graph is a graph that contains two vertices of different degrees.} graphs of bounded degree. 
A non-regular connected graph of maximum degree $\Delta$ is $(\Delta - 1)$-degenerate and therefore $\Delta$-choosable. With this in mind, Dvo\v{r}\'ak et al.\ asked the following question:

\begin{question}[\cite{Dvorak}]\label{q:non-regular}
  For each $\Delta \geq 2$, does there exist a value $\epsilon = \epsilon(\Delta) > 0$ such that any non-regular connected graph $G$ of maximum degree $\Delta$ is $\epsilon$-flexibly $\Delta$-choosable?
\end{question}

Later, in~\cite{choiclemen}, Choi et al.\ asked another specific question regarding degeneracy and flexibility.
Knowing that outer-planar graphs are $2$-degenerate, the authors asked:

\begin{question}[\cite{choiclemen}]\label{q:outer-planar}
  Are outer-planar graphs flexibly 3-choosable?
\end{question}

We will answer both questions in the affirmative.
\medskip

We dedicate Section~\ref{s:maxdeg} to solving Question~\ref{q:non-regular}.
We, in fact, show a stronger characterization for flexibility in connected graphs $G$ of maximum degree $\Delta$.
When $\Delta=2$, $G$ is flexibly $2$-choosable if and only if $G$ is a path (Theorem~\ref{thm:maxdeg2}). 
When $\Delta \geq 3$, $G$ is flexibly $\Delta$-choosable if and only if $G$ is not a $(\Delta + 1)$-clique (Theorem~\ref{thm:maxdeg3}).
In our proof, we use a seminal result by Erd\H{o}s, Rubin, and Taylor~\cite{ErdosDlist}, which characterizes graphs that can be list-colored whenever each vertex's color list has size equal to its degree. 
We prove theorems (Theorems~\ref{thm:weaker_lists_new} and \ref{thm:weaker_lists_general}) of a similar flavour that describes a sufficient condition for flexibility based on the degrees of the vertices in $G$.
Moreover, we provide an example (Figure~\ref{fig:weak_diamong}) that hints at which situations need to be avoided while aiming for a characterization of graphs that are flexibly choosable with lists of size equal to their vertex degree. 

We dedicate Section~\ref{s:tw} to solving Question~\ref{q:outer-planar}.
In fact, we prove the stronger statement that graphs of treewidth 2 are weighted flexibly 3-choosable (Theorem~\ref{thm:2tree}); hence, our result encompasses not only outer-planar graphs, but also series-parallel graphs and other graphs of treewidth $2$. 
Given a graph $G$ of treewidth $2$ with a $2$-treewidth decomposition, our method finds a list coloring on $G$ satisfying a fraction of $\frac{1}{3}$ of any weighted request in linear time. Furthermore, as a $k$-tree decomposition can be constructed in linear time for constant $k$~\cite{Bodlaender96}, our method therefore also runs in linear time. 
At the end of the section, we give a sufficient condition for lists of size $k+1$ that allow every weighted request on a $k$-tree to be $\frac{1}{k+1}$-satisfiable with respect to these lists (Theorem~\ref{thm:lambda_ch}). We state our result using a new concept of Zhu called $\lambda$-choosability~\cite{Zhu}.

Next, we append a short section (Section~\ref{s:td}) dedicated to the more restrictive graph parameter of treedepth. 
We show that graphs of treedepth $k$ are weighted $\frac{1}{k^2}$-flexibly $k$-choosable (Theorem~\ref{thm:td}).

In the last section (Section~\ref{s:degeneracy}), we propose a study of a new property stronger than flexibility that is motivated by degeneracy ordering.
We also explain relations between this property and a standard line of research concerning spanning trees with many leaves~\cite{Bonsma, Griggs}.
First, we give a standard definition of a \emph{$k$-degeneracy ordering}, which is an ordering of the vertices of a graph such that each vertex has at most $k$ neighbours appearing previously in the order.
Then, we say that a 
graph $G$ is \emph{$\epsilon$-flexibly $k$-degenerate} if for any subset $R \subseteq V(G)$, there exists a $k$-degeneracy ordering $D$ of $V(G)$ such that for at least $\epsilon|R|$ vertices $r \in R$, each neighbor of $r$ appears after $r$ in the ordering $D$.

For technical reasons, we will actually consider the slightly weaker notion of \emph{almost $\epsilon$-flexible $k$-degeneracy}, which we define formally in Section~\ref{s:degeneracy}, Definition~\ref{def:flexible_degeneracy}. 
We will show that for each $\Delta \geq 3$, there exists a value $\epsilon = \epsilon(\Delta) > 0$ such that if a non-regular graph $G$ of maximum degree $\Delta$ is 3-connected, then $G$ is almost $\epsilon$-flexibly $(\Delta-1)$-degenerate (Corollary~\ref{cor:degenerate}).

  We remark that in addition to Theorem~\ref{thm:2tree}, Theorem~\ref{thm:weaker_lists_new} can also be straightforwardly turned into a polynomial-time algorithm that finds a list coloring satisfying a positive fraction of coloring requests.
These algorithmic results can be compared with previous tools, which only give non-constructive proofs for the existence of an $\varepsilon$-satisifable coloring for an $\varepsilon$-flexible graph (as discussed in~\cite{Choi}).

\subsection{Preliminaries}
Let $G=(V,E)$ be a graph.
For an edge $e = \{uv\}\in E$, we say that a vertex $w\in V$ is adjacent to $e$ if $\{w u\}\in E$ or $\{w v\}\in E$.
We will always use $L$ to denote an assignment of color lists to each vertex of a graph $G$, and for a vertex $v \in V(G)$, we will use $L(v)$ to denote the list of colors assigned to $v$. 
We assume throughout the entire paper that all graphs are connected, as questions about colorings of a disconnected graph may be answered by analyzing each component separately.

We also define a slightly weaker version of weighted requests which we call \emph{uniquely weighted request}, where $w(v,c)$ can be nonzero only for at most one color $c\in L(v)$ for each vertex $v \in V(G)$.
This notion will be important only when we aim to calculate a specific (tight) value for our flexibility proportion, as for a general weighted request, one can disregard all but the largest weight request at each vertex while only losing an $|L(v)|$ factor in the computed flexibility proportion.
We formalize this idea in the following observation.
\begin{observation}\label{ob:unique}
  Let $G$ be a graph with list assignment $L$. 
  If every uniquely weighted request on $G$ is $\varepsilon$-satisfiable, then $G$ is weighted $\frac{\varepsilon}{\max_{v\in V(G)} |L(v)|}$-flexible with respect to $L$.
\end{observation}

We will make use of a lemma from Dvořák, Norin, and Postle~\cite{Dvorak} that serves as a useful tool and is easy to prove.
This lemma tells us that in order to show weighted $\epsilon$-flexibility in a graph $G$, we do not need to consider every possible request, and it is enough to find a distribution on colorings such that each individual vertex $v \in V(G)$ is colored by a given color $c \in L(v)$ with probability at least $\epsilon$.
\begin{lemma}[Lemma~3 in~\cite{Dvorak}]
\label{lem:Probability}
Let $G$ be a graph with a list assignment $L$.
Suppose there exists a probability distribution on $L$-colorings $\varphi$ of $G$ such that for every $v\in V(G)$ and $c\in L(v)$, 
\[
\Pr[\varphi(v)=c]\ge\varepsilon.
\]%
Then $G$ is weighted $\varepsilon$-flexible with respect to $L$.
\end{lemma}

\section{Graphs of Bounded Degree}\label{s:maxdeg}
In this section, we will investigate which graphs of maximum degree $\Delta$ are flexibly $\Delta$-choosable, for all integers $\Delta \geq 2$. 
While every non-regular graph of maximum degree $\Delta$ is $\Delta$-choosable, the complete graph $K_{\Delta+1}$ shows that not every $\Delta$-regular graph is $\Delta$-choosable.
In~\cite{ErdosDlist}, Erd\H{o}s, Rubin, and Taylor give a complete characterization of $\Delta$-regular graphs $G$ with $\ch(G) = \Delta$ and show that $K_{\Delta +1}$ is in fact the only $\Delta$-regular graph that is not $\Delta$-choosable. 
Furthermore, the authors give a characterization of the more general notion of \emph{degree choosability}, which is defined as follows. 
We say that a graph $G$ is \emph{degree choosable} if $G$ can be list colored for any assignment $L$ of lists such that $|L(v)|\geq \deg(v)$ for all $v\in V(G)$.
Erd\H{o}s, Rubin, and Taylor's characterization of degree choosable graphs is given in terms of the \emph{blocks} of a graph, where a block of a graph $G$ is defined as a maximal connected subgraph of $G$ with no cut-vertex.
By this definition, a block of a graph $G$ is either a cut-edge or a $2$-connected subgraph. 
The characterization of degree choosable graphs is as follows. 
Recall that in this paper we only consider connected graphs.

\begin{theorem} [\cite{ErdosDlist}]
A graph $G$ is degree choosable if and only if $G$ contains some block that is not a clique and is not an odd cycle.
\label{thmErdos}
\end{theorem}

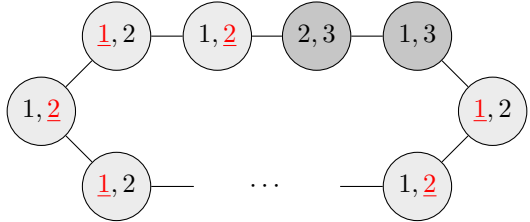
\begin{figure}
  \begin{center}
\begin{tikzpicture}
[scale=2,auto=left,every node/.style={circle,fill=gray!15}]
\node (a) at (0.33,1) [draw = black, fill = gray!45] {$2,3$};
\node (a2) at (-0.33,1) [draw = black] {$1,\textcolor{red}{\underline{2}}$};
\node (b1) at (0.4,0) [fill = white] {};
\node (b2) at (-0.4,0) [fill = white] {};

\node(dots) at (0,0) [fill = white] {$\cdots$};

\node (bb) at (1.5,0.5) [draw = black] {$\textcolor{red}{\underline{1}},2$};
\node (bbb) at (-1.5,0.5) [draw = black] {$1,\textcolor{red}{\underline{2}}$};
\node(c) at (1,1) [draw = black, fill = gray!45] {$1,3$};

\node(d) at (1,0) [draw = black] {$1,\textcolor{red}{\underline{2}}$};

\node(e) at (-1,1) [draw = black] {$\textcolor{red}{\underline{1}},2$};

\node(f) at (-1,0) [draw = black] {$\textcolor{red}{\underline{1}},2$};

\foreach \from/\to in {a/c,a/a2,c/bb,bb/d,d/b1,b2/f,f/bbb,bbb/e,e/a2}
    \draw (\from) -- (\to);

\end{tikzpicture}
\end{center}
\caption{The graph $G$ in the figure is an even cycle with color lists of size two. However, not even a single one of the red underlined requests can be satisfied.
  Therefore, the class of even cycles is not weakly flexibly $2$-choosable.}
\label{fig:cycle}
\end{figure}

The question of whether there exists a characterization of graphs that are flexibly degree choosable is open. However, we may straightforwardly show that Theorem \ref{thmErdos} does not give a characterization of flexible degree choosability. Indeed, Theorem \ref{thmErdos} implies that a $2$-regular graph $G$ is $2$-choosable if and only if $G$ is an even cycle. However, the following theorem, which characterizes flexible $2$-choosability in graphs of maximum degree $2$, shows that in general, cycles are not flexibly $2$-choosable.

\begin{theorem}\label{thm:maxdeg2}
Let $G$ be a graph of maximum degree $2$. Then $G$ is weakly flexibly $2$-choosable if and only if $G$ is a path. 
Furthermore, if $G$ is a path, $G$ is weighted $\frac{1}{2}$-flexibly $2$-choosable.
\end{theorem}
\begin{proof}
If $G$ is a path, then it is straightforward to show that there exist two list colorings $\phi_1, \phi_2$ on $G$ such that for each $v \in V(G)$, $L(v) = \{\phi_1(v), \phi_2(v)\}$. Then $G$ is weighted $\frac{1}{2}$-flexibly $2$-choosable by Lemma \ref{lem:Probability}. If $G$ is not a path, then $G$ is a cycle. If $G$ is an odd cycle, then $G$ is not $2$-choosable. If $G$ is an even cycle, then the color list assignment in Figure \ref{fig:cycle} shows that no positive fraction of coloring requests on $V(G)$ can be satisfied, even when a color is requested at every vertex of $G$.
\end{proof}

Thus, we see that Theorem~\ref{thmErdos} does not characterize graphs that are flexibly degree choosable, and the question of which graphs are flexibly degree choosable is open. 
However, we will give a complete characterization of graphs of maximum degree $\Delta \geq 3$ that are flexibly $\Delta$-choosable, which is a step toward characterizing flexible degree choosability. 
We will show that for a graph $G$ of maximum degree $\Delta \geq 3$, if $G \not \cong K_{\Delta + 1}$, then $G$ is flexibly $\Delta$-choosable. 
As $K_{\Delta+ 1}$ is not $\Delta$-choosable, this gives a complete characterization of the graphs of maximum degree $\Delta$ that are flexibly $\Delta$-choosable.
With our characterization, we answer Question~\ref{q:non-regular}.
First, it will be convenient to establish a corollary of Theorem \ref{thmErdos}.
\begin{corollary}
\label{corErdos}
Let $G$ be a graph, and let $L$ be a list assignment such that for each $v \in V(G)$, $|L(v)| \geq \deg(v)$. Then $G$ has an $L$-coloring if and only if either $G$ contains some block that is not a clique and is not an odd cycle, or $G$ contains a vertex $v$ for which $|L(v)| > \deg(v)$.
\end{corollary}
\toappendix{
\begin{proof}[Proof of Corollary~\ref{corErdos}]
If $|L(v)| = \deg(v)$ for every $v \in V(G)$, then the statement follows from Theorem \ref{thmErdos}. If there exists a vertex $v \in V(G)$ with $|L(v)| > \deg(v)$, then we may greedily color the vertices of $G$ in order of decreasing distance from $v$. In this process, each vertex $w \neq v$ will be colored before all neighbors of $w$ are colored, so as $|L(w)| \geq \deg(w)$, $w$ will have an available color. Furthermore, as $|L(v)| > \deg(v)$, $v$ will also have an available color. Therefore, $G$ is $L$-choosable.
\end{proof}
}
We are now ready to characterize graphs of maximum degree $\Delta \geq 3$ that are flexibly $\Delta$-choosable.

\begin{theorem}\label{thm:maxdeg3}
Let $G$ be a graph of maximum degree $\Delta \geq 3$. 
If $G \not \cong K_{\Delta+1}$, then $G$ is $\frac{1}{6 \Delta}$-flexibly $\Delta$-choosable, and $G$ is also weighted $\frac{1}{2 \Delta^4}$-flexibly $\Delta$-choosable.
\end{theorem}

Theorem \ref{thm:maxdeg3} is an improvement of our result which we presented in the conference version of this paper~\cite{ISAAC}. 
There, we proved an analogous result with a worse flexibility proportion of  $\frac{1}{2\Delta^3}$ for the unweighted case.
The new flexibility proportion of $\frac{1}{6 \Delta}$ in Theorem \ref{thm:maxdeg3} is considerably larger than the old value for all values of $\Delta \geq 3$. 
Moreover, observe that the flexibility proportion of $\frac{1}{6 \Delta}$ is best possible up to a constant factor for color lists of size $\Delta$.
Take, for example, a graph that consists of two copies of $K_\Delta$ joined by a perfect matching.
Now, consider a request that demands the same color at every vertex of one clique. 
Indeed, at most one request can be satisfied, giving an upper bound of $\frac{1}{\Delta}$ for the best possible flexibility proportion.

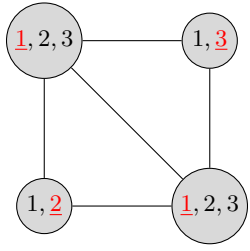
\begin{figure}
  \begin{center}
\begin{tikzpicture}
[scale=2.2,auto=left,every node/.style={circle,fill=gray!30},minimum size = 6pt,inner sep=2pt]
\node (a1) at (0,0) [draw = black] {\small{$1,\textcolor{red}{\underline{2}}$}};
\node (a2) at (1,0) [draw = black] {\small{$\textcolor{red}{\underline{1}},2,3$}};
\node (b1) at (0,1) [draw = black] {\small{$\textcolor{red}{\underline{1}},2,3$}};
\node (b2) at (1,1) [draw = black] {\small{$1,\textcolor{red}{\underline{3}}$}};

\foreach \from/\to in {a1/a2,b1/b2,a1/b1,a2/b2,a2/b1}
    \draw (\from) -- (\to);
\end{tikzpicture}
\end{center}
\caption{The graph shown here is not a clique or odd cycle and hence is list colorable with the given lists. However, not even a single vertex can be colored according to the widespread request shown by the red underlined colors.
Therefore, this example distinguishes choosability and weak flexibility.
Moreover, this example serves as an obstacle in further extensions of Theorems~\ref{thm:weaker_lists_new} and \ref{thm:weaker_lists_general}.
} 
\label{fig:weak_diamong}
\end{figure}

We first present here our simpler original proof with a worse bound which gives us the general weighted setting (Theorem~\ref{thm:weaker_lists_general}).
Then we show the improved version for the unweighted case (Theorem~\ref{thm:weaker_lists_new}).
Those results together imply Theorem~\ref{thm:maxdeg3}.
In both cases, we slightly extend Theorem \ref{thm:maxdeg3} further in the direction of Theorem~\ref{thmErdos} by allowing vertices of degree less than $\Delta(G)$ to have smaller lists.
With that being said, however, for each such vertex $v$ with $\deg(v) < \Delta(v)$, we require that $|L(v)| \geq \deg(v) + 1$. 
Hence, our sufficient condition for flexible choosability in Theorems \ref{thm:weaker_lists_general} and \ref{thm:weaker_lists_new} is likely not best possible.
On the other hand, Figure~\ref{fig:weak_diamong} provides some evidence that graphs that are not flexibly degree choosable may not be easily characterized.
In particular, Figure \ref{fig:weak_diamong} shows a diamond graph in which each vertex receives a color list of size equal to its degree along with a coloring request. 
However, in the figure, not a single coloring request can be satisfied. 
In contrast, Theorem \ref{thmErdos} shows that diamond graphs are degree choosable.

\begin{theorem}\label{thm:weaker_lists_general}
  Let $G$ be a graph of maximum degree $\Delta(G)\geq 3$, and let $\mathcal L$ consist of list assignments $L$ satisfying $|L(v)|\geq \deg(v)+1$ whenever $\deg(v)<\Delta(G)$ and $|L(v)|\geq \deg(v)$ whenever $\deg(v)=\Delta(G)$.
  If $G \not \cong K_{\Delta+1}$, then $G$ is weighted $\frac{1}{2 \Delta^4}$-flexibly $\mathcal L$-choosable.  
\end{theorem}

\begin{proof}
  In the proof, we will work with uniquely weighted requests, and we will aim to prove a flexibility proportion of $\frac{1}{2\Delta^3}$. 
Then, the result for weighted flexibility
follows from Observation~\ref{ob:unique}.

  Let $G$ be a connected graph that is not isomorphic to $K_{\Delta + 1}$. We assume that for each $v \in V(G)$, $L(v) \subseteq \mathbb{N}$. Let $R \subseteq V(G)$ be a set of vertices with coloring requests.
  As we consider only uniquely weighted requests, we can represent our coloring request with a function $f: R \rightarrow \mathbb{N}$.
Let each vertex $r \in R$ have a weight $w(r)$ that corresponds to the nonzero weight of the request at $r$.
For a subset $S \subseteq R$, let $w(S) = \sum_{r \in S} w(r)$.
As $\chi(G^3) \leq \Delta^3$, we may choose a set $R' \subseteq R$ of weight at least $\frac{1}{\Delta^3}w(R)$ with no two vertices within distance three of each other.
Note that $R'$ can be constructed greedily.
The next observation directly follows from our choice of $R'$:

\begin{claim}
Each edge of $G \setminus R'$ has at most one adjacent vertex in $R'$.
\label{obs}
\end{claim}

Now, suppose $A$ is a component of some possibly disconnected graph with a color list assignment $L'$ such that $|L'(v)| \geq \deg(v)$ for every $v \in V(A)$. We say that $A$ is a \textit{bad component} if $|L'(v)| = \deg(v)$ for every $v \in V(A)$ and every block of $A$ is either an odd cycle or a clique. If $A$ is not a \textit{bad component}, then we say that $A$ is a \textit{good component}.

We consider the graph $G \setminus R'$. We give $G \setminus R'$ a color list assignment $L'$ as follows. For a vertex $v \in V(G \setminus R')$, if $N_G(v) \cap R' = \emptyset$, then we let $L'(v) = L(v)$. If there exists a vertex $r \in N_G(v) \cap R'$, then we let $L'(v) = L(v) \setminus \{f(r)\}$. By Claim \ref{obs}, every vertex of $G \setminus R'$ has at most one neighbor in $R'$, and so $L'$ is well-defined. 
By Corollary \ref{corErdos}, if $G \setminus R'$ has no bad component, then $G$ may be $L$-colored in a way that satisfies our request at all of $R'$, in which case we satisfy a total weight of at least $\frac{1}{\Delta^3}w(R)$. Otherwise, let $A$ be a bad component of $G \setminus R'$.
We first observe that as $A$ is a bad component, for every vertex $v \in V(A)$, $|L'(v)| = \deg_{A}(v)$, and hence $\deg_G(v) = \Delta$.
By combining this fact with Claim~\ref{obs}, we obtain the following claim.

\begin{claim}\label{cl:bad}
If $A$ is a bad component, then for every vertex $v \in V(A)$, $\Delta-1\leq \deg_{A}(v)= |L'(v)| \leq \Delta$.
\end{claim}

We show that $A$ is not a single block. Indeed, suppose $A$ is a single block. 
Then, by Corollary \ref{corErdos}, $A$ must be a clique or odd cycle, and in particular, $A$ is a regular graph.
As $G$ is not isomorphic to $K_{\Delta + 1}$, $A \not \cong K_{\Delta + 1}$. 
Thus, by Claim~\ref{cl:bad}, %
$\deg_{A}(v)= \Delta - 1$ for each $v \in V(A)$, and hence $A$ must either be isomorphic to $K_{\Delta}$ or an odd cycle $C_{2k + 1}$ ($k \geq 2$) in the case that $\Delta = 3$. 
If $A \cong K_{\Delta}$, however, $A$ must have exactly one neighbor $r \in R'$ by Claim \ref{obs}, from which it follows that $\{r\} \cup A \cong K_{\Delta + 1}$, a contradiction. 
If $A \cong C_{2k + 1}$, $k \geq 2$, and $\Delta = 3$, then again by Claim \ref{obs}, $A$ must have a single neighbor $r \in R'$ which is adjacent to every vertex of $A$. This is a contradiction, as this implies that $\deg_G(r) \geq 5 > \Delta$. Therefore, $A$ has at least two blocks.

Now, we consider a \emph{terminal block} $B$ which is a leaf in the \emph{block-cut tree} of $A$ (c.f.~\cite[Chapter~3.1]{Diestel})---that is, $B$ is a block that only shares a vertex with one other block of $A$.
We claim that $B \cong K_{\Delta}$. To show this, we consider a vertex $v \in V(B)$ that is not a cut-vertex in $A$. If $\deg_A(v) = \Delta$, then $B \cong K_{\Delta + 1}$, which is a contradiction.
Hence, by Claim~\ref{cl:bad}, $\deg_A(v) = \Delta - 1$, and $B$ has a neighbor in $R'$. 
As $A$ is a bad component, Claim \ref{obs} implies that $B$ has exactly one neighbor $r \in R'$, which must be adjacent to every non cut-vertex of $B$. This implies that $|V(B)| \leq \Delta + 1$, which rules out the possibility that $B \cong C_{2k + 1}$ for some $k \geq 2$ when $\Delta = 3$. Then, as $A$ is a bad component, $B$ is $(\Delta - 1)$-regular, and it follows for all values $\Delta \geq 3$ that $B \cong K_{\Delta}$ and that $r$ is adjacent to $\Delta - 1$ vertices of $B$. We note that $r$ must then be adjacent to exactly one terminal block of a bad component, namely $B$. Furthermore, as the block-cut tree of $A$ has at least two leaves, $A$ has at least two terminal blocks $B,B'$ and hence two vertices $r, r' \in R'$ adjacent to $B, B'$ and no other terminal block of any bad component.

Now, we will construct a set $R^+ \subseteq R'$. As we construct $R^+$, we will define $R'' = R' \setminus R^+$. To construct $R^+$, for each bad component $A$ of $G \setminus R'$, we will choose a vertex $r \in R'$ of least weight adjacent to a terminal block of $A$, and we will add $r$ to $R^+$. Note that such a vertex $r$ has at least two neighbors $u,v \in V(A)$, and as $u,v$ belong to a terminal block of $A$, $uv$ must be an edge in a triangle $uvw$ of $A$ for which $w \not \sim r$. Therefore, $A \cup \{r\}$ contains an induced diamond subgraph, and by Theorem \ref{thmErdos}, $A \cup \{r\}$ is not contained in a bad component with respect to any color list assignment.

We also construct a color list assignment $L'':V(G \setminus R'') \rightarrow \mathbb{N}$ such that $G \setminus R''$ has no bad component with respect to $L''$. For a vertex $v \in V(G \setminus R'')$, if $N_G(v) \cap R'' = \emptyset$, then we let $L''(v) = L(v)$. If there exists a vertex $r \in N_G(v) \cap R''$, then we let $L''(v) = L(v) \setminus \{f(r)\}$. Again, by Claim \ref{obs}, $L''$ is well-defined. Any bad component of $G \setminus R''$ with respect to $L''$ must also be a bad component of $G \setminus R'$ with respect to $L'$, and hence by our choice of $R^+$, $G \setminus R''$ has no bad component with respect to $L''$. Therefore, by first coloring each vertex $r \in R''$ with $f(r)$ and then giving $G$ an $L''$-coloring by Theorem \ref{thmErdos}, we find an $L$-coloring on $G$ that satisfies a total request weight of at least $w(R'')$. As $w(R'') \geq \frac{1}{2}w(R') \geq \frac{1}{2 \Delta^3}w(R)$, the proof for uniquely weighted requests is complete. The general weighted flexibility statement then follows from Observation~\ref{ob:unique}.
\end{proof}

Now, we give an improved theorem for the unweighted setting.

\begin{theorem}\label{thm:weaker_lists_new}
  Let $G$ be a graph of maximum degree $\Delta(G)\geq 3$, and let $\mathcal L$ consist of list assignments $L$ satisfying $|L(v)|\geq \deg(v)+1$ whenever $\deg(v)<\Delta(G)$ and $|L(v)|\geq \deg(v)$ whenever $\deg(v)=\Delta(G)$.
  If $G \not \cong K_{\Delta+1}$, then $G$ is $\frac{1}{6 \Delta}$-flexibly $\mathcal L$-choosable.  
\end{theorem}

\begin{proof}
We define $L$, $R$, and $f$ as before.
As $\chi(G) \leq \Delta$, we may choose a set $R' \subseteq R$ of size at least $\frac{1}{\Delta}|R|$ that is an independent set.
Note that $R'$ can be constructed greedily if we aim for a slightly worse size of $\frac{1}{\Delta + 1} |R|$.
We define good components and bad components as before.
We consider the graph $G \setminus R'$. We give $G \setminus R'$ a color list assignment $L'$ as follows. For a vertex $v \in V(G \setminus R')$, we let
\[L'(v) = L(v) \setminus \{f(r): r \in N_G(v) \cap R'\}.\]
\smallskip
By Corollary \ref{corErdos}, if $G \setminus R'$ has no bad component with respect to $L'$, then $G$ may be $L$-colored in a way that satisfies our request at all of $R'$, in which case we satisfy at least $\frac{1}{\Delta}|R|$ requests. Otherwise, let $A$ be a bad component of $G \setminus R'$ with respect to $L'$.
We first observe that as $A$ is a bad component, for every vertex $v \in V(A)$, $|L'(v)| = \deg_{A}(v)$, and hence $|N_G(v) \cap R'| = \Delta - \deg_{A}(v)$. Therefore, if $A \cong K_{\Delta}$ or $A \cong K_1$, then the number of edges with an endpoint in $V(A)$ and an endpoint in $R'$ is exactly $\Delta$. Furthermore, if $A \not \cong K_{\Delta}$ and $A \not \cong K_1$, 
then $A$ is either a $K_t$ for $2 \leq t \leq \Delta - 1$, an odd cycle, or a graph with at least two terminal blocks, and in each of these cases, the number of edges with an endpoint in $V(A)$ and an endpoint in $R'$ is at least $2 \Delta- 2 > \Delta$.
Therefore, if 
$\beta$ is the number of bad components in $G\setminus R'$ with respect to $L'$, then 
$|R'| \geq  \beta.$
Furthermore, we observe that every terminal block of a bad component has at least $\Delta - 1$ incident edges with an endpoint in $R'$.

Now, we will recursively build a set $R^+$, and as we update $R^+$, we will also define $R'' = R' \setminus R^+$. We will also define a new list assignment 
\[L''(v) = L(v) \setminus \{f(r): r \in N_G(v) \cap R''\}.\]
Our aim will be to show that by adding certain vertices of $R'$ to $R^+$, we can eliminate all bad components of $G \setminus R''$ with respect to $L''$. From now on, whenever we refer to a good or bad component, we do so with respect to the list assignment $L''$.

We define some notation. Suppose we have some set $R^+ \subseteq R'$ already defined, and let $C_1$ be the number of bad components in $G \setminus R''$. For a vertex $r \in R''$, let $C_2(r)$ be the number of bad components in $G \setminus R''$ after adding $r$ to $R^+$. Then, we write $b(r) = C_1 - C_2(r)$. In other words, $b(r)$ is equal to the value by which the number of bad components in $G \setminus R''$ decreases after adding $r$ to $R^+$.
We make the following claim, which holds with respect to any set $R^+ \subseteq R'$.

\begin{claim}
\label{cl:makeGood}
For any vertex $r \in R''$ with a neighbor in a bad component of $G \setminus R''$, $b(r) \geq 1$.
\end{claim}

\noindent\textit{Proof of Claim~\ref{cl:makeGood}.}~
Suppose that $r$ has a neighbor in some bad component $A$ of $G \setminus R''$.
If $r$ has a neighbor in a good component of $G \setminus R''$, or if $\deg(r) < \Delta$, or if $r$ has a neighbor in a second bad component of $G \setminus R''$, then the claim clearly holds.
Hence, we may assume that $\deg(r) = \Delta$ and that all neighbors of $r$ belong to $A$. As $G \not \cong K_{\Delta + 1}$, either not all neighbors of $r$ belong to the same block of $A$, or all neighbors of $r$ belong to an odd cycle block of $A$. In both cases, $r$ contains three neighbors that do not form a triangle, and hence the block containing $r$ in $G \setminus (R'' \setminus \{r\})$ is neither a cycle nor a clique and hence must be a good component after $r$ is added to $R^+$. Thus, the bad component $A$ is removed from $G \setminus R''$ by adding $r$ to $R^+$, and hence $b(r) \geq 1$.
\hfill$\diamondsuit$\medskip

Now, we execute the following process \textbf{(P)} exhaustively:

\medskip
\begin{center}
\textbf{(P)} While there exists a vertex $r \in R''$ for which $b(r) \geq 2$, we add $r$ to $R^+$.
\end{center}
\medskip

\begin{claim}\label{cl:alpha}
  If we add at least $\frac{1}{6} \beta$ vertices to $R^+$ during the process {\normalfont \textbf{(P)}},
  then the theorem holds.
\end{claim}
\noindent\textit{Proof of Claim~\ref{cl:alpha}.}~
We suppose that during the process {\normalfont \textbf{(P)}}, $\alpha \beta$ vertices are added to $R^+$, where $\alpha \geq \frac{1}{6}$.
Then, after adding these vertices to $R^+$, the number of bad components in $G \setminus R''$ is at most $(1 - 2 \alpha)\beta$.
By Claim \ref{cl:makeGood}, we then may add at most $(1 - 2 \alpha)\beta$ additional vertices to $R^+$ so that $G \setminus R''$ contains only good components.
Then, we may give $G$ a list coloring in which every vertex of $R''$ is colored with its preferred color.
In this case, the proportion of vertices colored with their preferred colors may be estimated as follows, using the fact that $|R'| \geq  \beta$ for the last inequality:
\[\frac{|R''|}{|R'|} \geq \frac{|R'| - \alpha \beta - (1 - 2 \alpha ) \beta}{|R'|} = 1 - (1 - \alpha)\frac{\beta}{|R'|} \geq \alpha.\]

Therefore, at least $\alpha|R'| \geq \frac{\alpha}{\Delta}|R| \geq \frac{1}{6 \Delta} |R|$ 
vertices of $R$ are colored with their preferred colors. 
\hfill$\diamondsuit$\medskip

By Claim~\ref{cl:alpha}, we assume that at most $\frac{1}{6} \beta$ vertices are added to $R^+$ during the process \textbf{(P)}. Furthermore, after the process \textbf{(P)} is finished, there are no more vertices in $R''$ whose addition to $R^+$ reduces the number of bad components in $G \setminus R''$ by at least two.

At this point, we would like to compare the number of bad components in $G \setminus R''$ and the number of vertices in $R''$. To do this, we will consider an auxiliary bipartite graph $H$. One partite set of $H$ will be made up by the vertices of $R''$, and the other partite set $D$ of $H$ will have a vertex for each bad component of $G \setminus R''$. We often identify the vertices of $D$ with the bad components of $G \setminus R''$.

As the process \textbf{(P)} has finished, for a vertex $r \in R''$, $\deg_H(r) \leq 2$. Also note that $\deg_H(A) \geq 2$ for each $A \in D$, as every bad component other than $K_1$ and $K_{\Delta}$ must have at least $2 \Delta - 2$ incident edges with an endpoint in $R''$, and a $K_1$ or $K_{\Delta}$ bad component also must have at least two neighbors in $R''$.
If $\deg_H(r) = 2$, $r$ is not adjacent to any good component of $G \setminus R''$,  
and for each bad component $A$ to which $r$ is adjacent, $r$ belongs to a clique or odd cycle block in $A\cup\{r\}$.
Moreover, $\deg_G(r) = \Delta$, as otherwise, adding $r$ to $R^+$ would reduce the number of bad components in $G \setminus R''$ by at least two, a contradiction with the process \textbf{(P)}.
We make the following claim.
\begin{claim}\label{cl:structure}
  Let $r \in R''$, and suppose that $\deg_H(r)=2$. Let $A, A'$ be the bad components of $G \setminus R''$ to which $r$ is adjacent.
  Then either $\max\{\deg_H(A), \deg_H(A')\} \geq 3$, or $\deg_H(A) = \deg_H(A') =2$, and one of $A, A'$ has a neighbor of degree $1$ in $H$.
\end{claim}
\noindent\textit{Proof of Claim~\ref{cl:structure}.}~
We distinguish the following cases: 
\begin{enumerate}
    \item \label{caseOddCycle}
    The vertex $r$ belongs to an odd induced cycle $K$ of length at least $5$ in either $G[A \cup \{r\}]$ or $G[A' \cup \{r\}]$.
      Without loss of generality, we assume that $V(K)  \subseteq V(A) \cup \{r\}$. Since the process \textbf{(P)} is exhausted, we may assume that $K$ is a block in $G[A \cup \{r\}]$. 
      Therefore, since each vertex of $K \setminus \{r\}$ is either a cut-vertex of $A$ or has $\Delta - 2$ neighbors in $R'' \setminus \{r\}$, we may conclude by counting the edges between $A$ and $R'' \setminus \{r\}$ that $\deg_H(A) \geq 1 + \left \lceil \frac{4(\Delta - 2)}{\Delta} \right \rceil \geq 3$.
    \item \label{case:middleBlock} 
    The vertex $r$ belongs to a non-terminal block of $G[A \cup \{r\}]$. Then, $G[A \cup \{r\}]$ must have at least two terminal blocks, and therefore, the number of edges between $A$ and $R''$ must be at least $2 \Delta - 2 + |N(r) \cap A|$, and hence, $\deg_H(A) \geq 3$.
  \item\label{o:leaf}
  The vertex $r$ has a single neighbor $u$ in $A$. Let $u$ belong to a block $B$ of $A$.
  \begin{enumerate}
      \item If $\deg_A(u) = 0$, then clearly $u$ has at least $\Delta \geq 3$ neighbors in $R''$, and $\deg_H(A) \geq 3$.
      \item If $\deg_A(u) = 1$, then
        if $A$ has a terminal block $B'$ not containing $u$, $B'$ must be adjacent to at least $\Delta - 1 \geq 2$ vertices in $R'' \setminus \{r\}$. 
        Otherwise, $|V(A)| = 2$, and the neighbor $v$ of $u$ in $A$ must be adjacent to at least $\Delta - 1 \geq 2$ vertices of $R'' \setminus \{r\}$. In both cases, $\deg_H(A) \geq 3$.
      \item If $\deg_A(u) \geq 2$, then $u$ has at least two neighbors $v, w \in V(A)$. If $v$ does not belong to $B$, then $v$ and $w$ are separated by $u$ in $A$, and thus $A$ must have at least two terminal blocks, in which case $\deg_H(A) \geq 3$. On the other hand, if $v$ belongs to $B$, then $\deg_B(v) = \deg_B(w) = \deg_B(u) \leq \Delta - 1$; thus, if $\deg_H(A) = 2$, then there must exist a single additional vertex $r' \in R''$ adjacent to all vertices $x \in A$ for which $\deg_A(x) < \Delta$, and thus $G[A \cup \{r'\}]$ must have a block that is neither a clique nor an odd cycle. Therefore, as the process \textbf{(P)} is exhausted, $\deg_H(r') = 1$.     
        \end{enumerate}

  \item The vertex $r$ belongs to a terminal block $B$ in $\{r\} \cup A$, and $ | N(r) \cap V(A) | =t \geq 2$. Then, as the process \textbf{(P)} is exhausted, and assuming we are not in Case \ref{caseOddCycle}, we may assume that $B \cong K_t$. Then, $\deg_H(A) \geq \Delta - t + 1$. If $t \leq \Delta - 2$, then $\deg_H(A) \geq 3$. Otherwise, $t = \Delta - 1$, and we may apply Case (\ref{o:leaf}) to the other bad component $A'$ to which $r$ is adjacent.
\hfill$\diamondsuit$
\end{enumerate}
\medskip

We now aim to show that $|R''| \geq \frac{5}{4} |D|$.
In order to prove this, we will give each vertex in $R''$ charge $1$, and we will give each vertex in $D$ charge $- \frac{5}{4}$.
It suffices to show that $H$ has an overall nonnegative charge.
We will rearrange charges on $H$ as follows:

\begin{enumerate}
  \item\label{o:deg2n} For each vertex $A \in D$ with $\deg_H(A)= 2$ and $\{u,v\} =  N(A)$ satisfying $\deg_H(u)=\deg_H(v)=2$, we consider the following cases: 
\begin{enumerate}
  \item\label{o:deg2nc1} $N(u)=N(v)=\{A,A'\}$ for some $A' \in D$. By applying Claim \ref{cl:structure} to $u$, we know that $\deg(A')\geq 3$.
  Then, we let each of $A$ and $A'$ take a charge of $\frac{1}{2}$ from each neighbor, and then we let $A'$ give a charge of $\frac{1}{4}$ to $A$. Afterward, $A$ and $A'$ both have nonnegative charge.
  \item\label{o:deg2nc2} $N(u)\neq N(v)$. Then there are two vertices $A'\in N(u)$ and $A''\in N(v)$ such that $A,A',A''$ are all distinct.  We know by applying Claim \ref{cl:structure} to $u$ that either $\deg_H(A')\geq 3$, or $A'$ has a neighbor of degree $1$. Similarly, we know that $\deg_H(A'') \geq 3$, or $A''$ has a neighbor of degree $1$. Then, we let $A$, $A'$, and $A''$ each take a charge of $\frac{1}{2}$ from each neighbor of degree $2$ and a charge of $1$ from each neighbor of degree $1$; then we let $A$, $A'$, and $A''$ share these charges evenly.
  As the total charge taken by $A$, $A'$, and $A''$ is at least $4$, each of $A$, $A'$, and $A''$ ends with positive charge.
\end{enumerate}
  \item\label{o:deg3}
  For each remaining vertex $A \in D$ of negative charge with $\deg_H(A) \geq 3$, we let $A$ take a charge of $\frac{1}{2}$ from each of its neighbors, after which $A$ has positive charge.
  \item\label{o:deg2p}
  For each remaining vertex $A \in D$ of negative charge with $\deg_H(A) = 2$, if $A$ has a neighbor $v$ of degree $1$, then we let $A$ take a charge of $1$ from $v$ and a charge of $\frac{1}{2}$ from each of its other neighbors. Afterward, $A$ has positive charge.
\end{enumerate}

By Claim \ref{cl:structure}, this process gives all vertices $A \in D$ nonnegative charge. Furthermore, each vertex of $R''$ ends with nonnegative charge, as each degree $2$ vertex of $R''$ gives away at most a charge of $\frac{1}{2}$ to each neighbor, and each degree $1$ vertex of $R''$ gives away at most a charge of $1$ to its neighbor. Therefore, $H$ overall has nonnegative charge, and $|R''| \geq \frac{5}{4} |D|$.

By Claim~\ref{cl:makeGood}, adding at most $|D|$ additional vertices to $R^+$ makes all components of $G \setminus R''$ good, and then $G$ can be given a proper list coloring that satisfies the request of each remaining vertex of $R''$.

We estimate the total proportion of vertices in $R$ whose requests are satisfied during this process. We write $|R''|$ to denote the number of vertices in $R''$ immediately after the process \textbf{(P)} has ended. By the argument above, at least $|R''| - |D|$ vertices have their requests satisfied. As $|R''| \geq \frac{5}{4}|D|$, at least $\frac{1}{5} |R''|$ requests are satisfied. Furthermore, by Claim \ref{cl:alpha}, $|R''| \geq |R'| - \frac{1}{6} \beta$, giving us at least 
\[\frac{1}{5} \left (|R'| - \frac{1}{6} \beta \right )\]
satisfied requests. Finally, as $|R'| \geq \beta$, we must have at least 
\[\frac{1}{6} |R'| \geq \frac{1}{6 \Delta} |R|\]
satisfied requests. This completes the proof.
\end{proof}

\section{Graphs of Bounded Treewidth}\label{s:tw}
\sv{\toappendix{\section{Additions to Section~\ref{s:tw}}\label{sub:tw}}}

In this section, we consider graphs of bounded treewidth.
We characterize \textit{treewidth}\footnote{For an equivalent definition of treewidth using a tree decomposition, refer to e.g.~\cite{Bodlaender2006}.}
in terms of \textit{$k$-trees}, which is defined as follows. Given a nonegative integer $k$, a $k$-tree is a graph that may be constructed by starting with a $k$-clique and then repeatedly adding a vertex of degree $k$ whose neighbors induce a $k$-clique. The treewidth of a graph $G$ is then the smallest integer $k$ for which $G$ is a subgraph of a $k$-tree.
For technical reasons, we define a $0$-tree to be an independent set, which is an exception to our overall connectivity assumption. 

The class of connected graphs of treewidth $1$ is simply the class of trees.
The class of connected graphs of treewidth at most $2$ includes connected outer-planar graphs and series-parallel graphs, among other graphs. 
Graphs of bounded treewidth are of particular interest in the study of graph algorithms, as problems that are intractable in general are often tractable on graphs of bounded treewidth; for a survey on algorithmic aspects of treewidth, see~\cite{Bodlaender2006}. 
\lv{%

}%
As $k$-trees are $k$-degenerate, it follows that graphs of treewidth $k$ are $(k+1)$-choosable. The following result, shown implicitly in \cite{Dvorak}, shows furthermore that graphs of treewidth $1$ (i.e.\ trees) are $\frac{1}{2}$-flexibly $2$-choosable.
\begin{proposition}[\cite{Dvorak}]
\label{prop:Tree}
Let $G$ be a 1-tree with lists of size $2$.
Then there exists a set $\Phi = \{\phi_1, \phi_2\}$ of two proper colorings on $G$ such that for each vertex $v \in V(G)$, $\{\phi_1(v), \phi_2(v)\} = L(v)$. 
In particular, $G$ is weighted $\frac{1}{2}$-flexibly $2$-choosable.
\end{proposition}

In this section, we will show that graphs of treewidth $2$ are $\frac{1}{3}$-flexibly $3$-choosable (Theorem~\ref{thm:2tree}). 
We will show furthermore that for any positive integer $k$, if a graph $G$ of treewidth $k$ has a set $\mathcal L$ of list assignments of size $k+1$ that obey certain restrictions, then $G$ is $\frac{1}{k+1}$-flexibly $\mathcal L$-choosable (Theorem~\ref{thm:lambda_ch}).
By considering a $(k+1)$-clique whose vertices all have the same color lists in which the same color is requested at every vertex, we see that a $\frac{1}{k+1}$ flexibility proportion is best possible.

When we prove a result for graphs of treewidth $k$, we will only consider $k$-trees, as a proper coloring on a graph must also give a proper coloring for every subgraph. 

\begin{theorem}
\label{thm:2tree}
Let $G$ be a 2-tree with color lists of size $3$.
Then there exists a set $\Phi = \{\phi_1, \phi_2, \phi_3, \phi_4, \phi_5, \phi_6\}$ of six colorings on $G$ such that for each vertex $v \in V(G)$,
\lv{$\{\phi_1(v), \phi_2(v), \phi_3(v), \phi_4(v), \phi_5(v), \phi_6(v)\}$}
\sv{$\{\phi_1(v)$, $\phi_2(v), \phi_3(v),$ $\phi_4(v), \phi_5(v), \phi_6(v)\}$}
is a multiset in which each color from $L(v)$ appears exactly twice.
In particular, $G$ is weighted $\frac{1}{3}$-flexibly $3$-choosable.
\end{theorem}

It is interesting to note that the result of Theorem~\ref{thm:2tree} (as well as Proposition~\ref{prop:Tree}) can be easily turned into linear time algorithm that provides a $1/3$-satisfiable (resp.\ $1/2$-satisfiable) coloring, as the construction of the set $\Phi$ is algorithmic. 

\begin{proof}[Proof of Theorem~\ref{thm:2tree}]
We will construct a set of six $L$-colorings $\Phi: = \{\phi_1, \dots, \phi_6\}$ on $G$. Given an edge $uv \in E(G)$, we say that $\Phi$ is \textit{admissible} at $uv$ if the following conditions are satisfied:
\begin{itemize}
    \item For each $i = 1, \dots, 6$, $\phi_i(u) \neq \phi_i(v)$.
    \item If $\phi_i(u) = \phi_j(u)$ and $\phi_i(v) = \phi_j(v)$, then $i = j$.
    \item For each color $c \in L(u)$, $c$ appears exactly twice in the multiset $\{\phi_1(u), \dots, \phi_6(u)\}$.
    \item For each color $c' \in L(v)$, $c'$ appears exactly twice in the multiset $\{\phi_1(v), \dots, \phi_6(v)\}$.
\end{itemize}

We establish the following claim, which will be the main tool of our proof.
\sv{The proof of the claim is in the full version of the paper.}
\begin{claim}\label{cl:main}
Let $G$ be a $2$-tree, and let $uv \in E(G)$. Let $\Phi$ be a set of six $L$-colorings on $G$ that is admissible at every edge of $G$. Suppose a vertex $w$ is added to $G$ with neighbors $u,v$. Then $\Phi$ may be extended to $G+w$ so that $\Phi$ is also admissible at $uw$ and $vw$.
\label{claim:extend}
\end{claim}

\toappendix{
\noindent\textit{Proof of Claim~\ref{cl:main}.}~
Suppose that the claim does not hold. Let $L(u), L(v), L(w)$ give a counterexample where $|L(u) \cup L(v) \cup L(w)|$ is minimum. For our base case, if $L(u) = L(v) = L(w)$, then $\Phi = \{\phi_1, \dots, \phi_6\}$ on $G$ is uniquely defined at $uv$ due to the first two conditions of admissibility. Moreover, there is only one way to extend $\phi_1, \dots, \phi_6$ as proper colorings to $w$, hence there is only one way of extending $\Phi$ on $G+w$, and the extended $\Phi$ will satisfy all four conditions of admissibility at $uw$ and $vw$.

Next, we show that in a minimum counterexample, none of $L(u), L(v), L(w)$ contains a color that only appears on one of the lists. Suppose first that some color $c^* \in L(u)$ does not appear at $L(v) \cup L(w)$. As $\Phi$ is admissible at $uv$, there exists a unique element $x \in L(v)$ such that $(c^*, x) \neq (\phi_i(u), \phi_i(v))$ for each $i = 1, \dots, 6$. We temporarily replace $c^*$ with $x$ in $L(u)$ and in $\Phi$. Then $\Phi$ is still admissible at $uv$. (This substitution may cause $\Phi$ to be inadmissible at other edges in $G$ containing $u$, but this substitution is only temporary.) Furthermore, by the minimality of our counterexample, after this substitution, $\Phi$ may be extended to $w$ in such a way that $\Phi$ is admissible at $uw$ and $vw$. Now, we replace $x$ back by $c^*$ in $L(u)$ and in $\Phi$, and in this way we have extended $\Phi$ to $G+w$. By the choice of $c^*$ and $x$, it follows  that $\Phi$ is admissible at each edge of $G+w$, including $uw$ and $vw$. This shows that if $L(u)$ contains a unique color that does not appear in $L(v) \cup L(w)$, $G$ cannot be a minimum counterexample. By a similar argument, we may assume that each color of $L(v)$ appears on $L(u)\cup L(w)$.

To complete the argument, we next suppose that there exists a color $c^* \in L(w)$ that does not appear at $L(u), L(v)$. As in the previous case, we easily extend $\Phi$ by temporarily switching $c^*$ in $L(w)$. By the minimality of our counterexample, if we replace $c^*$ with any color $c \in L(u)$ that does not already appear at $L(w)$ (such a color must exist), we may extend $\Phi$ to $w$ in such a way that $\Phi$ is admissible at each edge of $G+w$. Then, after replacing $c$ with $c^*$ at $w$, if follows from the choice of $c^*$ and $c$ that $\Phi$ remains admissible at each edge of $G+w$. 

Therefore, each color of $L(u), L(v), L(w)$ must appear on at least two lists, and we are ready to show that no counterexample to the claim that is minimum with respect to $|L(u) \cup L(v) \cup L(w)|$ exists. 

Suppose that $|L(u) \cap L(v)| = 1$. Then $L(u) \cup L(v)$ contains four colors that only appear on one of  $L(u)$ or $L(v)$. Since $L(w)$ contains three colors,  some color $c^* \in L(u) \cup L(v)$ must only appear on one list, which leads to the case we already considered above. 

Suppose that $|L(u) \cap L(v)| = 2$. Then there exist two colors $c \in L(u)$ and $c' \in L(v)$ that each only appear once in $L(u) \cup L(v)$. If no uniquely appearing color exists in $L(u) \cup L(v) \cup L(w)$, then $c,c' \in L(w)$, and the third color of $L(w)$ must then appear in $L(u) \cap L(v)$. In this case, we may assume without loss of generality that $L(u) = \{1,2,3\}, L(v) = \{1,2,4\}, L(w) = \{1,3,4\}$, and also, as $\Phi$ is admissible at $uv$, that
\begin{eqnarray}
   \phi_1(u) = 1; & \phi_1(v) = 2 ;\nonumber \\
   \phi_2(u) = 1; & \phi_2(v) = 4 ;\nonumber \\
   \phi_3(u) = 2; & \phi_3(v) = 1 ;\nonumber \\
   \phi_4(u) = 2; & \phi_4(v) = 4 ;\nonumber \\
   \phi_5(u) = 3; & \phi_5(v) = 1 ;\nonumber \\
   \phi_6(u) = 3; & \phi_6(v) = 2. \nonumber 
\end{eqnarray}
Then, letting $\phi_1(w) = 4$, $\phi_2(w) = 3$, $\phi_3(w) = 3$, $\phi_4(w) = 1$, $\phi_5(w) = 4$, $\phi_6(w) = 1$, we extend $\Phi$ on $G+w$, and moreover, $\Phi$ is admissible at $uw$ and $vw$. 

Finally, suppose that $|L(u) \cap L(v)| = 3$. Then, if we assume that $L(w)$ does not contain any uniquely appearing color, we must have that $L(u) = L(v) = L(w)$, which is the base case which we have already considered. Thus, the claim holds.
\hfill$\diamondsuit$\medskip
}

Now, as $G$ is a $2$-tree, $G$ may be constructed by starting with a single edge and repeatedly adding a vertex of degree two whose two neighbors induce an edge. Suppose we start with an edge $e$. It is straightforward to construct an admissible set $\Phi$ of six colorings of $e$. Then, suppose that we have a partially constructed $2$-tree $G'$ and that $\Phi$ is an admissible coloring set at each edge of $G'$. We may add a vertex $w$ to $G'$ with adjacent neighbors $u,v$, and by Claim \ref{claim:extend}, we may extend $\Phi$ to $w$ while still letting $\Phi$ be admissible at every edge of the new graph. By this process, we may construct a set $\Phi$ of six colorings that is admissible at every edge of $G$.
\lv{%

}%
We conclude the proof by a simple application of Lemma~\ref{lem:Probability} on $\Phi$ with $\epsilon=1/3$.
\end{proof}

For $k \geq 3$, the question of whether or not $k$-trees are flexible with lists of size $k+1$ is still open.
However, after adding some restrictions to our color lists, we can guarantee the existence of a flexible list coloring of any $k$-tree with lists of size $k+1$. Our method will be an application of the algorithm of Theorem \ref{thm:2tree}.
In order to state our result precisely, we will need a definition. 

Given a partition $\lambda = \{\lambda_1, \dots, \lambda_t\}$ of $(k+1)$---that is, an integer multiset for which $\lambda_1 + \dots + \lambda_t = k+1$---a $\lambda$-assignment $L$ on a graph $G$ is a list assignment for which $\bigcup_{v \in V(G)} L(v)$ may be partitioned into parts $C_1, \dots, C_t$ such that for each $v \in V(G)$ and each value $1 \leq i \leq t$, $|L(v) \cap C_i| = \lambda_i$. Then, a graph $G$ is \emph{$\lambda$-choosable} if there exists a list coloring on $G$ for any $\lambda$-assignment $L$. Zhu introduces $\lambda$-assignments in \cite{Zhu} and notes 
that for any integer $k$, $\{k\}$-choosability is equivalent to $k$-choosability, and $\underbrace{\{1,1, \dots, 1\}}_\text{$k$ times}$-choosability is equivalent to $k$-colorability. With this definition, $\lambda$-choosability gives a notion of colorings that lie between traditional colorings and traditional list colorings. Zhu shows, for example, that while tripartite planar graphs are not $4$-choosable in general, these graphs are always $\{1,3\}$-choosable. Choi and Kwon remark furthermore that while general planar graphs are not $\{1,3\}$-choosable, the question of whether all planar graphs are $\{1,1,2\}$-choosable is still open \cite{Choi}. We may extend the concept of $\lambda$-choosability to flexible list colorings by saying that a graph $G$ is $\epsilon$-flexibly $\lambda$-choosable if, given any $\lambda$-assignment $L$ and request $r$ on $G$, $r$ is $\epsilon$-satisfiable with respect to $L$. Then we have the following theorem.

\begin{theorem}\label{thm:lambda_ch}
Let $G$ be a $k$-tree, and let $\lambda = \{\lambda_1, \dots, \lambda_t\}$ be a partition of $k+1$ with parts of size at most $3$. Then $G$ is $\frac{1}{k+1}$-flexibly $\lambda$-choosable.
\end{theorem}

\toappendix{
  \sv{\begin{proof}[Proof of Theorem~\ref{thm:lambda_ch}]}
\lv{\begin{proof}}
We will induct on $t$, the number of parts in $\lambda$. We will show that there exists a set $D$, $|D| \in (k+1) \mathbb{Z}$, of $L$-colorings on $G$ such that for any given vertex $v \in V(G)$ and any given color $c \in L(v)$, $c$ appears at $v$ in exactly $\frac{1}{k+1}|D|$ colorings in $D$. 

When $t = 1$, $G$ is a $k$-tree, where $k = \lambda_1 - 1$, and $\lambda_1 \leq 3$. Our $\lambda$-assignment $L$ is simply an assignment of color lists of size $\lambda_1  =k+1$. Therefore, our task is to give $G$ a $\frac{1}{\lambda_1}$-flexible $L$-coloring. If $\lambda_1 = 1$, then $k = 0$ and $G$ is a single vertex, in which case the statement is trivial. Otherwise, if $\lambda_1 = 2$ or $\lambda_1 = 3$, then the statement either follows from Proposition \ref{prop:Tree} or Theorem \ref{thm:2tree}.

Next, suppose $t \geq 2$. 
As $G$ is a $k$-tree, we may order the vertices of $V(G)$ as $v_1, \dots, v_n$ so that the following properties are satisfied:
\begin{itemize}
\item $v_1, \dots, v_k$ form a clique.
\item For any $i \geq k+1$, $v_i$ has exactly $k$ neighbors $v_j$, $j < i$, and these neighbors form a $k$-clique.
\end{itemize}
Given a vertex $v_i \in V(G)$, we say that a vertex $v_j \in N(v_i)$, $j < i$ is a \textit{back-neighbor} of $v_i$.

We will construct a set of $L$-colorings of $G$ as follows. We define the family of vertex subsets 
\[ \mathcal A: = {{\{v_1, \dots, v_{k}\}} \choose {k - \lambda_t + 1}} \cup {{\{v_1, \dots, v_{k}\}} \choose {k - \lambda_t}}.\]

For each vertex subset $A \in \mathcal A$, we will construct a set of $L$-colorings of $G$. We will then obtain $D$ by taking the union of all of these colorings.

For each vertex subset $A \in \mathcal A$, we will recursively construct a set $S_A \subseteq V(G)$. We start by letting $S_A = A$. Next, we consider the vertices $v_{k+1}, \dots, v_n$ in order. For each vertex $v_i$, $i \geq k + 1$, if $v_i$ has exactly $k-\lambda_t$ back-neighbors in $S_A$, then we add $v_i$ to $S_A$; otherwise, we exclude $v_i$ from $S_A$. After we have considered all vertices $v_i \in V(G)$ in this way, our set $S_A \subseteq V(G)$ is successfully constructed. Figure \ref{fig:3tree} shows an example construction of a set $S_A$ when $k = 3$ and $\lambda_t = 1$.

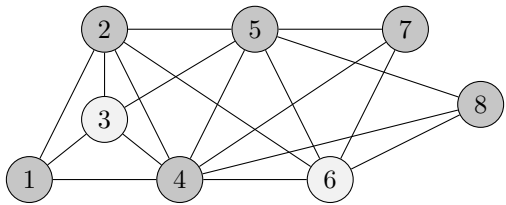
\begin{figure}
  \begin{center}
\begin{tikzpicture}
[scale=2,auto=left,every node/.style={circle,fill=gray!15}]
\node (v1) at (0,0) [draw = black, fill = gray!45] {$1$};
\node (v2) at (0.5,1) [draw = black, fill = gray!45] {$2$};
\node (v3) at (0.5,0.4) [draw = black, fill = gray!10] {$3$};
\node (v4) at (1,0) [draw = black, fill = gray!45] {$4$};
\node (v5) at (1.5,1) [draw = black, fill = gray!45] {$5$};
\node (v6) at (2,0) [draw = black, fill = gray!10] {$6$};
\node (v7) at (3,0.5) [draw = black, fill = gray!45] {$8$};
\node (v8) at (2.5,1) [draw = black, fill = gray!45] {$7$};
\foreach \from/\to in {v1/v2,v1/v3,v1/v4,v2/v3,v2/v4,v3/v4,v5/v4,v5/v3,v5/v2,v6/v5,v6/v4,v6/v2,v7/v6,v7/v5,v7/v4,v8/v5,v8/v6,v8/v4}
    \draw (\from) -- (\to);

\end{tikzpicture}
\end{center}
\caption{The figure shows a $3$-tree $G$ with its vertices numbered according to their position in the $3$-tree ordering of $V(G)$. The dark vertices show the set $S_A$ that is obtained from our construction when $r = 1$ and $A = \{1,2\}$. $S_A$ is constructed as follows. As $4$ has exactly two back-neighbors in $S_A$, $4$ is added to $S_A$, and $5$ is added similarly. However, as $6$ has three back-neighbors in $S_A$, $6$ is not added to $S_A$. Vertices $7$ and $8$ are both added to $S_A$, as they both have exactly two back-neighbors in $S_A$. The resulting vertices of $S_A$ form a $2$-tree.
}
\label{fig:3tree}
\end{figure}

\begin{claim}
\label{claim:cliques}
For any $k$-clique $K \subseteq V(G)$, the family of vertex subsets $\{V(K) \cap S_A: A \in \mathcal A\}$ exactly contains every $(k-\lambda_t+1)$-subset and every $(k-\lambda_t)$-subset of $V(K)$. That is, 
\[\{V(K) \cap S_A: A \in \mathcal A\} = {V(K) \choose {k - \lambda_t + 1}} \cup {V(K) \choose {k - \lambda_t}}.\]
\end{claim}
\noindent\textit{Proof of Claim~\ref{claim:cliques}.}~
The claim clearly holds for the clique induced by $v_1, \dots, v_k$, so it is enough to show that the claim still holds each time we add a new vertex $v_i$, $i \geq k+1$.

Suppose we have just added a vertex $v_i$. Let $N_b(v_i)$ denote the back-neighbors of $v_i$---that is, the neighbors $v_j \in N(v_i)$ for which $j < i$. The set $\{v_i\} \cup N_b(v_i)$ induces a $k+1$ clique, and hence $G[\{v_i\} \cup N_b(v_i)]$ contains $k+1$ $k$-cliques as subgraphs. We assume that the claim already holds for $G[N_b(v_i)]$, so we only need to show that for each $a \in N_b(v_i)$, the claim holds for the clique $K = G[\{v_i\} \cup N_b(v_i) \setminus \{a\}]$.

We will show that for any such subset $B \subseteq V(K)$ of size $k-\lambda_t$ or $k-\lambda_t+1$, there exists a subset $B' \subseteq N_b(v_i)$ of size $k-\lambda_t$ or $k-\lambda_t+1$ such that when $N_b(v_i) \cap S_A = B'$, our rule for constructing $S_A$ extends $B' \cap K$ to $B$ in $K$.

Given a subset $B \subseteq V(K)$ of size $k-\lambda_t$ or $k-\lambda_t+1$:
\begin{itemize}
\item If $|B| = k-\lambda_t$ and $v_i \in B$, then $S_A \cap V(K) = B$ if and only if $B' = B \setminus \{v_i\} \cup \{a\}$.
\item If $|B| = k-\lambda_t+1$ and $v_i \in B$, then $S_A \cap V(K) = B$ if and only if $B' = B \setminus \{v_i\}$.
\item If $|B| = k-\lambda_t$ and $v_i \not \in B$, then $S_A \cap V(K) = B$ if and only if $B' = B \cup \{a\}$.
\item If $|B| = k-\lambda_t+1$ and $v_i \not \in B$, then $S_A \cap V(K) = B$ if and only if $B' = B$.
\end{itemize}
Hence, for each $(k-\lambda_t+1)$ or $(k-\lambda_t)$-subset $B \subseteq V(K)$, there exists a subset $B' \subseteq N_b(v_i)$ of size $k-\lambda_t$ or $k-\lambda_t+1$ such that when $N_b(v_i) \cap S_A = B'$, our rule for constructing $S_A$ extends $B' \cap K$ to $B$ in $K$. Thus we see that for any $k$-clique $K$ of $G$ and for any $(k-\lambda_t+1)$ or $(k-\lambda_t)$-subset $B \subseteq V(K)$, $B$ appears in the set $\{V(K) \cap S_A: A \in \mathcal A\}$, and it must follow that each $B$ appears exactly once. \hfill$\diamondsuit$\medskip

Now, we fix a vertex subset $A \in \mathcal{A}$ and a corresponding set $S_A$, and we partition $G$ into graphs $G_1:= G[S_A]$ and $G_2:=G[V(G) \setminus S_A]$. We claim that $G_1$ is a $(k-\lambda_t)$-tree. Indeed, consider a vertex $v_i \in V(G_1)$. If $i \leq k$ and $|A| = k - \lambda_t$, then $v_i$ belongs to the $(k-\lambda_t)$-clique induced by $A$, which has no back-neighbors. If $i \leq k$ and $|A| = k-\lambda_t+1$, then either $v_i$ belongs to a $(k-\lambda_t)$-clique contained in $A$ with no back-neighbors, or $v_i$ is the highest indexed vertex of $A$, in which case the back-neighbors of $v_i$ induce a clique $A \setminus \{v_i\}$. Otherwise, if $i > k$, then $v_i$ was added to $S_A$ because $v_i$ had exactly $k-\lambda_t$ back-neighbors in $S_A$. Furthermore, as $N_b(v_i)$ induces a clique in $G$, it follows that these $k-\lambda_t$ back-neighbors of $v_i$ form a $(k-\lambda_t)$-clique. Therefore, $G_1$ is a $(k-\lambda_t)$-tree. By a similar argument, $G_2$ is an $(\lambda_t-1)$-tree.

Next, with $S_A$ fixed, we construct a distribution of colorings on $G$. First, we use the induction hypothesis to choose a set $D'$ of $L'$-colorings on $G_1$, where $L'(v) = L(v) \cap (C_1 \cup \dots \cup C_{t-1})$ for each $v \in V(G_1)$.
As $G_2$ is a $(\lambda_t-1)$-tree, for any fixed coloring $\phi \in D'$, we may extend $\phi$ to a set of $\lambda_t!$ colorings on $G$ by giving $G_2$ $\lambda_t!$ list colorings by Proposition \ref{prop:Tree} or Theorem \ref{thm:2tree}. This gives us a total of $\lambda_t!|D'|$ colorings of $G$ for our fixed set $S_A$. Hence, putting together the colorings we obtain from each set $S_A$, we obtain a set $D$ containing $\left ({{k -1} \choose {\lambda_t-1}} + {{k - 1} \choose \lambda_t} \right )\lambda_t!|D'|$ colorings of $G$.

For a vertex $v \in V(G)$, and a color $c \in L(v)$, we count the number of colorings of $D$ in which $c$ appears at $v$. We fix a $k$-clique $K$ to which $v$ belongs. If $c \not \in C_t$, then $c$ is used at $v$ only in those colorings computed from the subsets $S_A$ for which $v \in S_A \cap V(K)$. By Claim \ref{claim:cliques}, there are exactly ${{k -1} \choose {\lambda_t-1}} + {{k - 1} \choose \lambda_t}$ sets $S_A$ for which $v \in S_A \cap V(K)$, and by the induction hypothesis, for each such set $S_A$, we compute exactly $\frac{\lambda_t! |D'|}{k - \lambda_t + 1}$ colorings in which $c$ is used at $v$. Therefore, the number of colorings of $D$ in which $c$ is used at $v$ is equal to $\frac{\lambda_t! |D'|}{k - \lambda_t + 1} \left ({{k -1} \choose {\lambda_t-1}} + {{k - 1} \choose \lambda_t} \right ) = |D'| \cdot \frac{k!}{(k - \lambda_t + 1)!}$. 

Next, for a color $c \in L(v)$ with $c \in C_t$, $c$ is used at $v$ only in those colorings of $D$ computed from subsets $S_A$ for which $v \not \in S_A \cap V(K)$. There are exactly ${k - 1 \choose {\lambda_t-2}} + {k - 1 \choose {\lambda_t-1}}$ such sets $A$ for which $v \not \in S_A \cap V(K)$ (with ${{k - 1} \choose {-1}} = 0)$, and for each such set $S_A$, we compute exactly $(\lambda_t-1)!|D'|$ colorings in which $c$ is used at $v$. Therefore, the number of colorings in which $c$ is used at $v$ is exactly $(\lambda_t-1)!|D'| \left ({k - 1 \choose {\lambda_t-2}} + {k - 1 \choose {\lambda_t-1}} \right ) = |D'| \cdot \frac{k!}{(k - \lambda_t + 1)!}$. Therefore, for each vertex $v \in V(G)$ and each color $c \in L(v)$, $c$ is used at $v$ in exactly $|D'| \cdot \frac{k!}{(k - \lambda_t + 1)!}$ colorings. 

Finally, we show that $G$ has a flexible $L$-coloring. Suppose we choose a random coloring $\phi \in D$ of $G$. For a vertex $v \in V(G)$ and a color $c \in L(v)$, $\phi(v) = c$ with probability 
\[\frac{|D'| \cdot \frac{k!}{(k - \lambda_t + 1)!}}{\lambda_t! |D'| \left ({k \choose {\lambda_t-1}} + {k \choose \lambda_t} \right )} = \frac{1}{k+1}.\]
Hence, by Lemma \ref{lem:Probability}, $G$ is $\frac{1}{k+1}$-flexible with respect to the list assignment $L$.
\end{proof}
}

\sv{The proof is in the full version of the paper.}
We conclude the section by noting that by improving Theorem \ref{thm:2tree}, one can also improve Theorem \ref{thm:lambda_ch}. In particular, suppose we could prove for some $k_0 \geq 3$ and all $k \leq k_0$ that for a $k$-tree $G$ and a list assignment $L$ of $k+1$ colors at each $v \in V(G)$, that there exists a set $\Phi$ of $(k+1)!$ colorings on $G$ such that each color of each list $L(v)$ appears at $v$ exactly $k!$ times in $\Phi$. (This statement for $k = 2$ is exactly the statement of Theorem \ref{thm:2tree}.) Then, we could relax the requirement of Theorem \ref{thm:lambda_ch} to allow parts of $\lambda$ of size at most $k_0$. However, even proving this statement for $k = 3$ seems like a difficult problem.

\section{Graphs of Bounded Treedepth}\label{s:td}
\sv{\toappendix{\section{Additions to Section~\ref{s:td}}\label{sub:td}}}

In this section, we will consider graphs of bounded treedepth. For a rooted tree $T$, we define the \emph{height} of $T$ as the number of vertices in the longest path from the root of $T$ to a leaf of $T$. Then, the \emph{treedepth} \td$(G)$ of a graph $G$ is defined as the minimum height of a rooted tree $T$ for which $G\subseteq \text{Closure}(T)$, where the \emph{closure} of a rooted tree $T$, written $\text{Closure}(T)$, is a graph on $V(T)$ in which each vertex is adjacent to all of its ancestors in $T$ and all of its descendants in $T$.

If \td$(G)= k$, then $G$ is $(k-1)$-degenerate, as each leaf of the corresponding tree has at most $k-1$ ancestors and no descendants.
It follows that such a graph is $k$-choosable, and the complete graph on $k$ vertices shows us that this is best possible.
We will show that graphs of treedepth $k$ are not only $k$-choosable, but weighted $\varepsilon$-flexibly $k$-choosable as well, with $\epsilon = \frac{1}{k}$ in the unweighted case, and $\epsilon = \frac{1}{k^2}$ in the weighted case.
Note that $K_k$ also shows that our value of $\epsilon=\frac{1}{k}$ that we obtain for the unweighted case is the best possible.

\lv{Before we show our main proof, we develop a variant of Lemma~\ref{lem:Probability} that is suited to uniquely weighted requests.
The difference is that this modification allows us to build a different distribution for each request.}
\sv{Before we show our main proof, we develop a variant of Lemma~\ref{lem:Probability} that is suited to uniquely weighted requests and allows us to build a different distribution for each request. With this new lemma, }
\lv{Moreover,}
we can weaken the assumption on $\phi$ where we only require $\Pr[\varphi(v)=c]\ge\varepsilon$ for $c$ such that $w(c,v)\neq 0$.
As we are working with uniquely weighted requests, for each $v\in V(G)$ there is at most one color $c$ that we need to use at $v$ with positive probability.
We often refer to this color $c$ as \emph{the requested color} at $v$.
\sv{The proof of Lemma~\ref{lem:Probability2} is very similar to the original proof of Lemma~\ref{lem:Probability}.
We include the proof in the full version of the paper.}

\begin{lemma}
\label{lem:Probability2}
Let $G$ be a graph with list assignment $L\in \mathcal L$.
Let $R\subseteq V(G)$ be a set of vertices with uniquely weighted requests given by a  weight function $w$.
Suppose that for any weighted request $w$ on $G$, there exists a probability distribution on $L$-colorings $\varphi$ of $G$ such that for every $v\in V(G)$, $c \in L(v)$ such that $w(v,c)\neq 0$,
\[%
\Pr[\varphi(v)=c]\ge\varepsilon.
\]%
Then $G$ is $\varepsilon$-flexibly $\mathcal L$-choosable.
\end{lemma}

\lv{
The proof is very similar to the original proof of Lemma~\ref{lem:Probability}.
We include it for completeness.
}

\toappendix{
\begin{proof}[Proof of Lemma~\ref{lem:Probability2}]
  Let $G$, $L$, and $w$ be as described in the statement and fixed.
  Let $\phi$ be  chosen at random based on the given probability distribution.
By linearity of expectation:
\[\mathbf{E}\Biggl[\sum_{v\in V(G)} w(v,\phi(v)) \Biggr]=\sum_{\substack{v\in V(G) \\ c\in L(v) \\ w(v,c)\neq 0}}    \Pr[\phi(v)=c]\cdot w(v,c) \ge\varepsilon \cdot \sum_{\substack{v\in V(G) \\c\in L(v)}} w(v,c), 
\]
and thus there exists the required $L$-coloring $\phi$ satisfying at least an $\epsilon$ fraction of the request.
\end{proof}
}

\begin{theorem}\label{thm:td}
Let $G$ be a graph of treedepth $k$. Then $G$ is $\frac{1}{k}$-flexibly $k$-choosable.
In fact, the same holds even when uniquely weighted requsts are considered.
Therefore, $G$ is weighted $\frac{1}{k^2}$-flexibly $k$-choosable.
\end{theorem}
\begin{proof}
We will prove that $G$ is weighted $\frac{1}{k}$-flexibly $k$-choosable with respect to some arbitrary assignment of uniquely weighted requests. We will inductively construct a coloring distribution on $G$ and then apply Lemma~\ref{lem:Probability2}.
\lv{%

}%
We induct on $k$. When $k = 1$, the statement is trivial.

Suppose that $k > 1$.
Let $G$ be a subgraph of the closure of a tree $T$ of height $k$ and root $v$. 
First, we color $v$ with a color $c \in L(v)$ uniformly at random.
Then, we delete $c$ from all other lists in $G$. 
For any vertex $u \in V(G)$ whose list still has $k$ colors, we arbitrarily delete another color from $L(u)$, taking care not to delete a requested color, if one exists, at $u$.
Now, we obtain a coloring distribution on the remaining vertices of $G$ by the induction hypothesis, which is possible, as each component of $G \setminus v$ is a graph of treedepth $k-1$ with lists of size $k-1$.
For a vertex $w \neq v$ at which a color $a \in L(w)$ is requested, the probability that $a$ is assigned to $w$ is at least the probability that $a$ is not deleted from $L(w)$ multiplied by the probability that $w$ is assigned the color $a$ by the induction hypothesis.
The probability that $a$ is not deleted from $L(w)$ is at least $\frac{k-1}{k}$, as the requested color $a$ can only be deleted from $L(w)$ by using $a$ at $v$. 
The probability that $a$ is used at $L(w)$ by the induction hypothesis is $\frac{1}{k-1}$. 
Therefore, $a$ is used at $w$ with probability at least $\frac{k-1}{k} \cdot \frac{1}{k-1} = \frac{1}{k}$. 
Furthermore, the probability that the requested color $c \in L(v)$ is used at $v$, if such a request exists, is exactly $\frac{1}{k}$. Hence, by  
Lemma~\ref{lem:Probability2}, the proof is complete.
The weighted flexibility statement then follows from Observation~\ref{ob:unique}.
\end{proof}

\section{Flexible Degeneracy Orderings}\label{s:degeneracy}
\sv{\toappendix{\section{Additions to Section~\ref{s:degeneracy}}\label{sub:degeneracy}}}
\lv{Given a graph $G$ and an integer $k$, we define a \emph{$k$-degeneracy order} on $V(G)$ as an ordering $D$ of $V(G)$ such that for each vertex $v \in V(G)$, at most $k$ neighbors of $v$ appear before $v$ in $D$.}
\sv{Recall that a}\lv{A} graph $G$ with a $k$-degeneracy ordering is called $k$-degenerate. 
One of the earliest appearances of graph degeneracy is in a paper by Erd\H{o}s and Hajnal \cite{ErdosColoring}, in which the authors define the \textit{coloring number} $\col(G)$ of a graph $G$ as one more than the minimum integer $k$ for which $G$ is $k$-degenerate. The coloring number of $G$ satisfies
$\ch(G) \leq \col(G)$, and hence an upper bound on a graph's coloring number implies an upper bound on a graph's choosability. \sv{Similarly, the concept of \emph{flexible degeneracy}, defined in the introduction, extends the notion of coloring number to the setting of flexibility.}

\lv{To extend the notion of coloring number to the setting of flexibility,
we say that a graph $G$ is $\epsilon$-flexibly $k$-degenerate if for any subset $R \subseteq V(G)$ of ``requested'' vertices, there exists a $k$-degeneracy ordering $D$ of $V(G)$ such that for at least $\epsilon|R|$ vertices $r \in R$, all neighbors of $r$ appear after $r$ in $D$. Flexible degeneracy is related to flexible choosability by the fact that a graph that is $\epsilon$-flexibly $k$-degenerate is $\epsilon$-flexibly $(k+1)$-choosable.}

 \lv{However, we see immediately that if a graph $G$ has a single vertex $w$ of degree at most $k$, then in any $k$-degenerate ordering of $V(G)$, $w$ must appear as the last vertex. Therefore, if we allow $R = \{w\}$, then $G$ is not $\epsilon$-flexibly $k$-degenerate for any value $\epsilon > 0$.}
 \sv{Suppose we wish to determine if a $k$-degenerate graph is $\epsilon$-flexibly $k$-degenerate for some $\epsilon > 0$. If $G$ has a single vertex $w$ of degree at most $k$, then in any $k$-degenerate ordering of $V(G)$, $w$ must appear as the last vertex. Therefore, if we allow our request $R$ to contain the single vertex $w$, then we see that $G$ is not $\epsilon$-flexibly $k$-degenerate for any value $\epsilon > 0$.}
Thus, in order to avoid this small problem, we give the following definition:

\begin{definition}
\label{def:flexible_degeneracy}
Let $G$ be a graph. For a constant $\epsilon > 0$ and an integer $k \geq 1$, we say that $G$ is \emph{almost $\epsilon$-flexibly $k$-degenerate} if the following holds. Let $R \subseteq V(G)$ be a vertex subset, and if $G$ contains a single vertex $w$ of degree at most $k$, let $R \neq \{w\}$. Then there exists an ordering $D$ on $V(G)$ such that 
\begin{itemize}
    \item For each vertex $v \in V(G)$, at most $k$ neighbors of $v$ appear before $v$ in $D$.
    \item There exist at least $\epsilon |R|$ vertices $r \in R$ for which no neighbor of $r$ appears before $r$ in $D$.
\end{itemize}
\end{definition}

In this section, we will investigate flexible degeneracy in non-regular graphs of maximum degree $\Delta$. We will show that for each $\Delta \geq 3$, there exists a value $\epsilon = \epsilon(\Delta) > 0$ such that if a non-regular graph $G$ of maximum degree $\Delta$ is 3-connected, then $G$ is almost $\epsilon$-flexibly $(\Delta - 1)$-degenerate.

 Given a non-regular graph $G$ of maximum degree $\Delta$ and a set $R$ of requested vertices, the problem of finding a flexible degeneracy ordering on $G$ with respect to $R$ is closely related to the problem of finding a spanning tree on $G$ whose leaves intersect a positive fraction of the vertices in $R$. We formalize this observation with the following lemma.

\begin{lemma}
\label{lem:leaves_may_be_satisfied}
Let $G$ be a non-regular graph of maximum degree $\Delta$ with a vertex $w \in V(G)$ of degree less than $\Delta$. Let $T$ be a spanning tree of $G$, and let $\mathcal I$ be an independent set in $G$ such that $w \not \in \mathcal I$ and every vertex in $\mathcal I$ is a leaf of $T$. Then there exists a $(\Delta - 1)$-degeneracy order $D$ on $V(G)$ in which each vertex $v \in \mathcal I$ appears in $D$ before each neighbor of $v$.
\end{lemma}

\toappendix{
\sv{\begin{proof}[Proof of Lemma~\ref{lem:leaves_may_be_satisfied}]}
\lv{\begin{proof}}
We construct a walk $W$ on $G$. We let $W$ begin at $w$. We then let $W$ visit every vertex of $V(G) \setminus \mathcal I$, and then we finally let $W$ visit each vertex of $\mathcal I$. This is possible, as $T \setminus \mathcal I$ spans $G \setminus \mathcal I$.

From $W$, we may obtain a degeneracy order $D$ on $V(G)$ as follows. First, we let $D$ be a sequence containing the single vertex $w$. Then, we consider the walk $W$ one step at a time, and each time $W$ visits a new vertex $v \in V(G)$, we add $v$ to the beginning of the sequence $D$. As $W$ visits all vertices of $V(G)$, $D$ gives a degeneracy order on $V(G)$.

We claim that $D$ is a $(\Delta - 1)$-degeneracy order of $V(G)$. To show this, we consider a vertex $v \in V(G)$. The number of neighbors of $v$ that come before $v$ in $D$ is equal to the number of neighbors of $v$ that $W$ visits after $v$. If $v = w$, then $v$ has at most $\Delta - 1$ neighbors, so the number of $v$-neighbors appearing before $v$ in $D$ is clearly at most $\Delta - 1$. If $v \neq w$, then $v$ must have been reached by $W$ by passing through some neighbor of $v$. Therefore, at least one neighbor of $v$ comes after $v$ in the sequence $D$. Then, as the degree of $v$ is at most $\Delta$, at most $\Delta - 1$ neighbors of $v$ come before $v$ in $D$. Therefore, $D$ is a $(\Delta - 1)$-degeneracy order of $V(G)$. Furthermore, as $\mathcal I$ is an independent set, for each vertex $v \in \mathcal I$, $W$ visits $v$ after visiting each neighbor of $v$, so $v$ appears in $D$ before each of its neighbors. This completes the proof.
\end{proof}
}

Lemma \ref{lem:leaves_may_be_satisfied} tells us that given a non-regular graph $G$ of bounded degree, one way to find a flexible degeneracy order on $G$ is to find a spanning tree $T$ on $G$ in which many requested vertices are leaves in $T$. 
The following definition describes essentially how readily a graph $G$ may accommodate an arbitrary set of requested vertices as leaves of a spanning tree on $G$.

\begin{definition}
\label{def:game_connectivity}
Let $G$ be a graph. We define the \textit{game connectivity} $\kappa_g(G)$ of $G$ as follows. Given a spanning tree $T$ of $G$, we let $\Lambda (T)$ represent the set of leaves in $T$. Then, given a vertex subset $R \subseteq V(G)$, we define $l(R)$ to be the maximum value of $\frac{|R \cap \Lambda (T)|}{|R|}$, maximized over all spanning trees $T$ of $G$. Finally, we define
\[%
\kappa_g(G) = \min_{R \subseteq V(G)} l(R).
\]%
\end{definition}

\lv{
We may think of game connectivity in terms of a \emph{Landscaper-Arborist} one-round game, which we define as follows. The game is played on a finite graph $G$ with two players: a Landscaper and an Arborist. 
First, Landscaper chooses a vertex subset $R \subseteq V(G)$. Then, Arborist chooses a spanning tree $T$ on $G$ using as many vertices of $R$ as leaves in $T$ as possible. Arborist then receives a score of $\frac{|R \cap \Lambda (T)| }{ |R|}$. Landscaper's goal is to minimize the score of Arborist, and Arborist's goal is to maximize his score.
}
Previous research considers the problem of finding a spanning tree in a graph with a large fraction of vertices as leaves \cite{Bonsma} \cite{Griggs}, as well as the problem of finding a spanning tree with leaves at prescribed vertices \cite{Egawa}.
\lv{The Landscaper-Arborist game hence is a natural combination of these two ideas. }
\sv{Game connectivity hence is a natural combination of these two ideas.}
Furthermore, similarly to Lemma \ref{lem:leaves_may_be_satisfied}, a degeneracy order on a graph may be obtained from a spanning tree of bounded degree. Hence, we observe that a tree of bounded degree whose leaves make up a large fraction of its vertices gives a notion of weakly flexible degeneracy. 
\lv{The Landscaper-Arborist game may equivalently be described as a \emph{Chooser-Remover game}, played with a Chooser and a Remover, in which Chooser chooses a vertex subset $R \subseteq V(G)$, and then Remover chooses a subset $R' \subseteq R$ such that each vertex of $R'$ has a neighbor in $V(G) \setminus R'$ and such that $G \setminus R'$ is a connected graph. By this equivalent definition, $\kappa_g(G)$ indeed gives a certain game theoretic notion of vertex connectivity.
}Note that game connectivity can also be expressed in the language of connected dominating sets. 
The next lemma shows that if the game connectivity of a non-regular $G$ of maximum degree $\Delta$ is bounded below, then $G$ is almost flexibly $(\Delta-1)$-degenerate.

\begin{lemma}
Let $G$ be a non-regular graph of maximum degree $\Delta \geq 3$. Then $G$ is almost $\frac{1}{2 \Delta} \kappa_g(G)$-flexibly $(\Delta - 1)$-degenerate.
\label{lem:game_connectivity}
\end{lemma}

\toappendix{
\sv{\begin{proof}[Proof of Lemma~\ref{lem:game_connectivity}]}
\lv{\begin{proof}}
By the hypothesis, $G$ has a vertex $w$ of degree at most $\Delta - 1$. We let $R_0 \subseteq V(G)$. As stated before, we will assume by the definition of almost flexible degeneracy that if every vertex of $V(G) \setminus \{w\}$ is of degree $\Delta$, then $R_0 \neq \{w\}$. Therefore, either $R_0$ contains at least two vertices, or the label $w$ may be chosen to represent vertex of degree at most $\Delta - 1$ that does not belong to $R_0$. In either case, we may define $R = R_0 \setminus \{w\} $, and $|R| \geq \frac{1}{2}|R_0|$ holds.

We may choose some subset $R' \subseteq R$ of vertices and a spanning tree $T$ on $G$ such that each vertex of $R'$ is a leaf in $T$. By the definition of game-connectivity, we may choose $R'$ such that $|R'| \geq \kappa_g(G)|R| \geq \frac{1}{2}\kappa_g(G)|R_0|$. Then, as $\chi(G) \leq \Delta $, we may choose an independent subset of $R'$ of size at least $\frac{1}{2 \Delta}\kappa_g(G)|R_0|$, and then the lemma follows from Lemma \ref{lem:leaves_may_be_satisfied}.
\end{proof}
}

Lemma \ref{lem:game_connectivity} shows that certain classes of graphs with robust game connectivity have flexible degeneracy orderings. However, calculating $\kappa_g(G)$ for an arbitrary graph $G$ does not appear to be an easy problem\sv{.}\lv{,
and naively calculating the value $l(R)$ from Definition \ref{def:game_connectivity} for each vertex subset $R \subseteq V(G)$ would require exponential time.}
Therefore, it will be useful to find general lower bounds for $\kappa_g(G)$ in graphs of bounded degree. 
However, the example in Figure \ref{fig:diamonds} shows that the game connectivity of a graph of bounded degree may be arbitrarily small, even when the graph is regular and of arbitrary degree.
\lv{We will see that by also requiring 3-connectivity in addition to an upper bound on vertex degree, we may obtain a lower bound on a graph's game connectivity.}
\sv{By also requiring 3-connectivity in addition to a bound on vertex degree, we will be able to obtain a lower bound on a graph's game connectivity.} %

\begin{figure}
  \begin{center}
\begin{tikzpicture}
[scale=1.7,auto=left,every node/.style={circle,fill=gray!15}]
\node (a1) at (1.6,0.4) [draw = black] {};
\node (a2) at (1.6,-0.4) [draw = black, fill = black] {};
\node (a3) at (1.8,0) [draw = black] {};
\node (a4) at (1.4,0) [draw = black] {};

\node (b1) at (-1.6,0.4) [fill = white] {};
\node (b2) at (-1.6,-0.4) [fill = white] {};

\node(dots) at (-1.6,0) [fill = white] {$\cdots$};

\node(c1) at (0.49,1.6) [draw = black] {};
\node(c2) at (1.16,1.2) [draw = black, fill = black] {};
\node(c3) at (0.7,1.2) [draw = black] {};
\node(c4) at (0.95,1.6) [draw = black] {};

\node(d1) at (-0.49,1.6) [draw = black, fill = black] {};
\node(d2) at (-1.16,1.2) [draw = black] {};
\node(d3) at (-0.7,1.2) [draw = black] {};
\node(d4) at (-0.95,1.6) [draw = black] {};

\node(e1) at (0.49,-1.6) [draw = black, fill = black] {};
\node(e2) at (1.16,-1.2) [draw = black] {};
\node(e3) at (0.7,-1.2) [draw = black] {};
\node(e4) at (0.95,-1.6) [draw = black] {};

\node(f1) at (-0.49,-1.6) [draw = black] {};
\node(f2) at (-1.16,-1.2) [draw = black, fill = black] {};
\node(f3) at (-0.7,-1.2) [draw = black] {};
\node(f4) at (-0.95,-1.6) [draw = black] {};

\foreach \from/\to in {a1/a3,a1/a4,a3/a4,a3/a2,a4/a2,c2/c3,c2/c4,c3/c4,c3/c1,c4/c1,d2/d3,d2/d4,d3/d4,d3/d1,d4/d1,e2/e3,e2/e4,e3/e4,e3/e1,e4/e1,f2/f3,f2/f4,f3/f4,f3/f1,f4/f1,a1/c2,c1/d1,f1/e1,e2/a2, d2/b1,b2/f2}
    \draw (\from) -- (\to);

\end{tikzpicture}
\end{center}
\caption{The graph $G$ in the figure is an arbitrarily large two-connected $3$-regular graph. If a set $R \subseteq V(G)$ is chosen as shown by the dark vertices in the figure, then there does not exist a constant $\epsilon > 0$ such that $\epsilon|R|$ vertices of $R$ may become leaves of some spanning tree of $G$. For any $k \geq 3$, a similar $k$-regular graph without flexible degeneracy may be constructed from a cycle $C$ by replacing each vertex of $C$ by a clique minus an edge.}
\label{fig:diamonds}
\end{figure}
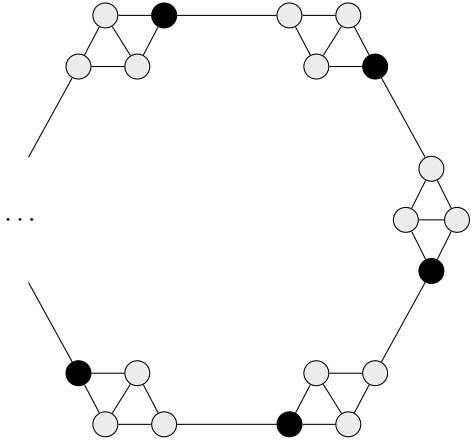

\lv{In the following results, we} \sv{We}
will show that for non-regular 3-connected graphs of maximum degree $\Delta$, there exists a constant $\epsilon= \epsilon(\Delta) > 0$ for which such graphs are almost $\epsilon$-flexibly $(\Delta - 1)$-degenerate.
\sv{ Our proof method may be applied to a more general hypergraph theorem, which we state and prove in the full version. Theorem~\ref{thm:game_connectivity}, and hence Corollary~\ref{cor:degenerate}, follow immediately. 
}%
\lv{Our proof method may be applied to a more general hypergraph problem, which we state in the following lemma. 
}%
\toappendix{
  For an integer $k \geq 0$, we say that a hypergraph $\mathcal H$ is $k$-\emph{edge-connected} if for every nonempty proper vertex subset $U \subsetneq V(\mathcal H)$, there exist at least $k$ edges of $\mathcal H$ containing at least one vertex of $U$ and at least one vertex of $V(\mathcal H) \setminus U$.

\begin{lemma}
\label{lem:hypergraph}
Let $d \geq 2$ be an integer. Then there exists a constant $\epsilon = \epsilon(d)$ such that the following holds.
Let $\mathcal H$ be a $3$-edge-connected hypergraph, and for each edge $e \in E(\mathcal H)$, let $|e| \leq d$. Then $\mathcal H$ has a connected spanning set of at most $(1 - \epsilon)|E(\mathcal H)|$ edges.  Furthermore, we may choose $\epsilon(d)$ to be the solution to the recurrence $\epsilon(2) = \frac{1}{3}$, $  \epsilon(d) = \frac{3\epsilon(\lceil \frac{d}{3} + 1 \rceil)}{d + 3\epsilon(\lceil \frac{d}{3} + 1 \rceil)}$ for $d \geq 3$.
\end{lemma}
}

\toappendix{
\lv{\begin{proof}}
\sv{\begin{proof}[Proof of Lemma~\ref{lem:hypergraph}]}
We induct on $d$. We let $V = V(\mathcal H)$ and $E = E(\mathcal H)$. We let $n = |V|$.

For $d = 2$, $\mathcal H$ is a graph of minimum degree $3$ with possible loops. As $\mathcal H$ is $3$-edge-connected, $|E| \geq \frac{3}{2}n$, and there exists a spanning tree $T$ on $V$ consisting of $n-1$ edges. Therefore, we may find a connected spanning set of edges on $\mathcal H$ using fewer than $\frac{2}{3}|E|$ edges, and hence the lemma holds for $d=2$ with $\epsilon(2) = \frac{1}{3}$.

Next, suppose that $d \geq 3$. We will build a set $A \subseteq E$ such that $(V,A)$ is a connected hypergraph. We will construct $A$ as a union of two parts $A_1, A_2$, and we will begin with $A_1 = \emptyset$. Note that initially, $(V,A_1)$ has $n$ components. We will construct $A_1$ through the following process. If there exists an edge $e \in E$ that intersects at least $\lceil \frac{d}{3} + 2 \rceil$ components of $(V,A_1)$, then we add $e$ to $A_1$. We repeat this process until no new edges of $E$ can be added to $A$ in this way. Let $c$ be the number of components of $(V,A_1)$ at this point. As each edge added to $A_1$ reduces the number of components of $(V,A_1)$ by at least $\frac{d}{3} + 1$, it follows that $c \leq n - (\frac{d}{3} + 1)|A_1|.$

We will write $k = \frac{3}{d}$, so that $|E| \geq kn$, and $c \leq n - \frac{1}{k}(1 + k)|A_1|$. We will consider two cases. We will show that in both cases, the lemma holds with $\epsilon = \epsilon(d) = \frac{3 \epsilon'}{d + 3 \epsilon'}$, where $\epsilon' = \epsilon(\lceil \frac{d}{3} + 1 \rceil)$. Thus, we will show inductively that $\epsilon$ is a positive constant depending only on $d$.

For the first case, suppose that $|A_1| = \alpha k n $ for some value $\alpha \geq \frac{k\epsilon' - k + 1}{1 + k \epsilon '}$ if $d \geq 4$ or that $\alpha \geq \frac{1}{4}$ if $d = 3$. Then, as $\mathcal H$ is a connected hypergraph, we may choose a set $A_2$ of $c-1$ edges such that $A = A_1 \cup A_2$ is connected.

We estimate the proportion of edges of $E$ not included in $A$. This proportion is 
\[\epsilon = \frac{|E| - |A_1| - |A_2|}{|E|}.\]
We apply the inequalities $|A_2| < c$ and $|E| \geq kn$.
\[\epsilon >  \frac{kn - |A_1| - c}{kn}.\]
Next, we apply the inequality $c \leq n - \frac{1}{k}(1 + k)|A_1|$.
\[\epsilon > \frac{kn - |A_1| - n + \frac{1}{k}(1 + k)|A_1|}{kn} = \frac{n(k - 1) + \frac{1}{k}|A_1|}{kn}.\]
Next, we apply the equation $|A_1| = \alpha k n$.
\[\epsilon > \frac{(k - 1) + \alpha}{k}.\]
When $d = 3$, we have $k = 1$, which immediately tells us that $\epsilon > \frac{1}{4}$, in which case the lemma holds; otherwise, we assume that $d \geq 4$.
As $\alpha \geq \frac{k\epsilon' - k + 1}{1 + k \epsilon '}$, we have that 
\begin{eqnarray*}
\epsilon &>& \frac{k- 1+\alpha}{k} \geq \frac{k-1}{k} +  \frac{\epsilon' - 1 + \frac{1}{k}}{1 + k \epsilon '} \\
& =& \frac{ (1 + \frac{3}{d}\epsilon')(1 - \frac{d}{3}) + \epsilon' - 1 + \frac{d}{3}}{1 + \frac{3}{d} \epsilon'}
= \frac{3 \epsilon'}{d + 3 \epsilon'}.
\end{eqnarray*}
This completes the first case.

Suppose, on the other hand, that $|A_1| < \frac{k\epsilon' - k + 1}{1 + k \epsilon '} \cdot kn$ if $d \geq 4$, or that $|A_1| < \frac{1}{4}kn$ if $d = 3$. Then, for any edge $e \in E \setminus A_1$, $e$ intersects at most $\lceil \frac{d}{3} + 1 \rceil$ distinct components of $(V,A_1)$. Therefore, if we contract all components of $(V,A_1)$ to single vertices, we obtain a hypergraph $\mathcal H'$ that satisfies the statement of the lemma but with $d$ replaced by $\lceil \frac{d}{3} + 1 \rceil < d$. (We allow the resulting hypergraph $\mathcal H'$ to have edges containing a single vertex, as well as multiple edges.) We know that \[|E(\mathcal H')| = |E| - |A_1| > |E| - \frac{k\epsilon' - k + 1}{1 + k \epsilon '} \cdot kn \geq \frac{k}{1 + k \epsilon '}|E|\] (or $|E(\mathcal H')| > \frac{3}{4} |E|$ when $d = 3$). Furthermore, by the induction hypothesis, we may find a connected set of spanning edges $A'$ on $\mathcal H'$ that avoids at least $\epsilon'|E(\mathcal H')|$ edges. Therefore, we may obtain $\mathcal H$ from $\mathcal H'$ by ``un-contracting'' edges, and we hence may extend $A'$ to a connected spanning set $A \subseteq E$ of edges on $\mathcal H$ that avoids at least $\epsilon' \cdot \frac{k}{1 + k \epsilon '}|E| =\frac{3 \epsilon'}{d + 3 \epsilon'}|E| $ edges (or $\frac{3}{4}\epsilon'|E|$ edges when $d = 3$). This completes the second and final case.
\end{proof}
}

\lv{
With Lemma \ref{lem:hypergraph}, we may give a lower bound for the game connectivity of 3-connected graphs of bounded degree, which will ultimately establish an almost flexible degeneracy result for these graphs.}

\begin{theorem}
\label{thm:game_connectivity}
Let $\Delta \geq 3$ be an integer. Then there exists a value $\epsilon = \epsilon(\Delta) > 0$ for which the following holds. Let $G$ be a $3$-connected graph of maximum degree $\Delta$. Then $\kappa_g(G) \geq \epsilon$. 
\end{theorem}

\begin{proof}
We choose an independent set $R' \subseteq R$ satisfying $|R'| \geq \frac{1}{\Delta + 1}|R|$. We construct a hypergraph $\mathcal H$ from $G$ and $R'$. We let $V(\mathcal H)$ be given by the components of $G \setminus R'$, and we let $E(\mathcal H)$ be given by $R'$. For a vertex $r \in R'$, we say that the corresponding edge $r \in E(\mathcal H)$ contains the vertices corresponding to the components to which $r$ is adjacent in $G \setminus R'$.

We claim that $\mathcal H$ is $3$-edge-connected. Indeed, consider a nonempty proper vertex subset $S \subsetneq V(\mathcal H)$, and consider the subgraph $G' \subseteq G$ containing the components of $S$ along with the vertices of $R'$ all of whose neighbors are in $S$. As $S \neq V(\mathcal H)$, $S$ must avoid some component in $G$. Therefore, as $S$ does not contain all components of $G \setminus R'$, by the $3$-connectivity of $G$, $G'$ must have at least three adjacent vertices of $R'$ that in turn are adjacent to components of $G \setminus R'$ not included in $G'$. These vertices of $R'$ correspond to at least three hyperedges of $\mathcal H$ that contain both a vertex of $S$ and a vertex not belonging to $S$. Furthermore, as $G$ has maximum degree $\Delta$, no vertex of $R'$ is adjacent to more than $\Delta$ components of $G \setminus R'$, and hence no hyperedge of $H$ contains more than $\Delta$ vertices. Thus we see that $\mathcal H$ satisfies the properties of the Lemma \ref{lem:hypergraph} with $\Delta$ as $d$.

By Lemma \ref{lem:hypergraph}, there exists a value $\epsilon = \epsilon(\Delta) > 0$ for which we may find some set of at most $(1 - \epsilon)|E(\mathcal H)| = (1 - \epsilon)|R'|$ hyperedges on $\mathcal H$ that give a connected spanning structure on $\mathcal H$. Therefore, we may choose a set $R^+ \subseteq R'$ of size at most $(1 - \epsilon)|R'|$ vertices such that $G \setminus (R' \setminus R^+)$ is connected. We will write $R'' = R' \setminus R^+$. We then may choose a spanning tree $T$ on $G \setminus R''$, and we may extend $T$ to all of $V(G)$ by letting the vertices of $R''$ be leaves of $T$. In this way, the leaves of our tree $T$ intersect $R'$ in at least $|R''|\geq \epsilon|R'|$ vertices. As $|R'| \geq \frac{1}{\Delta+ 1}|R|$, the leaves of $T$ intersect $R$ in at least $\frac{\epsilon}{\Delta+ 1}|R|$ vertices, and $\frac{\epsilon}{\Delta+ 1}$ is a positive constant dependent only on $\Delta$. This completes the proof.
\end{proof}

\begin{corollary}
\label{cor:degenerate}
Let $\Delta \geq 3$ be an integer. 
Then there exists a value $\epsilon(\Delta) > 0$
for which the following holds. Let $G$ be a $3$-connected non-regular graph of maximum degree $\Delta$.
Then $G$ is almost $\epsilon$-flexibly $(\Delta-1)$-degenerate.
\end{corollary}
\lv{\begin{proof}
The statement follows immediately from Theorem \ref{thm:game_connectivity} and Lemma \ref{lem:game_connectivity}.
\end{proof}
}

\lv{The graph in Figure \ref{fig:diamonds} shows that the 3-connectivity condition of Theorem \ref{thm:game_connectivity} may not be replaced by two-connectivity. We note, however, that the 3-connectivity requirement of Theorem \ref{thm:game_connectivity} may be weakened by instead requiring that every subset $S \subseteq V(G)$ of at least $k$ vertices have at least three vertices in its boundary, for some $k \geq 1$.
To prove that this weakened condition is sufficient, we may use the proof of Theorem \ref{thm:game_connectivity} exactly as it is written, with the exception that we choose $R' \subseteq R$ of size at least $\frac{1}{(\Delta+1)^{k+1}}|R|$ so that the vertices of $R'$ have a mutual distance of at least $k + 2$. In this way, each component of $G \setminus R'$ must have size of at least $k$, and then the rest of the proof may be applied as normal.
This weakened connectivity condition gives a constant $\epsilon$ that depends only on $\Delta$ and $k$. Furthermore, it is not difficult to show that Theorem \ref{thm:game_connectivity} also holds for non-regular 3-connected graphs of bounded degree whose edges have been subdivided at most $t$ times, for some $t \geq 0$. In this case, we obtain a constant $\epsilon$ depending only on $\Delta$ and $t$.}%
\lv{%

}%
\sv{The graph in Figure \ref{fig:diamonds} shows that the 3-connectivity condition of Theorem \ref{thm:game_connectivity} may not be replaced by 2-connectivity. }%
Furthermore, the following example shows that the bounded degree condition of Theorem \ref{thm:game_connectivity} may not be removed. Consider a graph $G$ with an independent subset $R \subseteq V(G)$ satisfying the following properties. 
For each triplet $A \in {R \choose 3}$, let there exist a vertex of $G \setminus R$ whose neighbors of $R$ are exactly those vertices from $A$.
Figure \ref{fig:counterexample} shows such a construction with $|R| = 5$.
It is straightforward to show that when $|R| \geq 4$, $G$ is 3-connected. However, no more than two vertices of $R$ may be removed from $G$ without disconnecting $G$. Therefore, for any spanning tree $T$ on $G$, the leaves of $T$ include at most two vertices of $R$. As $R$ may be arbitrarily large, this example shows that 3-connected graphs $G$ do not satisfy $\kappa_g(G) \geq \epsilon$ for any universal $\epsilon > 0$.

\begin{figure}
  \begin{center}
\begin{tikzpicture}
[scale=1.5,auto=left,every node/.style={circle,fill=gray!15}]
\node (r1) at (0,0) [draw = black, fill = black] {};
\node (r2) at (1,0) [draw = black, fill = black] {};
\node (r3) at (2,0) [draw = black, fill = black] {};
\node (r4) at (3,0) [draw = black, fill = black] {};
\node (r5) at (4,0) [draw = black, fill = black] {};
\node (u1) at (0,1) [draw = black, fill = gray!15] {};
\node (u2) at (1,1) [draw = black, fill = gray!15] {};
\node (u3) at (2,1) [draw = black, fill = gray!15] {};
\node (u4) at (3,1) [draw = black, fill = gray!15] {};
\node (u5) at (4,1) [draw = black, fill = gray!15] {};
\node (d1) at (0,-1) [draw = black, fill = gray!15] {};
\node (d2) at (1,-1) [draw = black, fill = gray!15] {};
\node (d3) at (2,-1) [draw = black, fill = gray!15] {};
\node (d4) at (3,-1) [draw = black, fill = gray!15] {};
\node (d5) at (4,-1) [draw = black, fill = gray!15] {};

\foreach \from/\to in {u1/r1,u1/r2,u1/r3,d1/r1,d1/r2,d1/r4,u2/r1,u2/r2,u2/r5,d2/r1,d2/r3,d2/r4,u3/r1,u3/r3,u3/r5,d3/r1,d3/r4,d3/r5,u4/r2,u4/r3,u4/r4,d4/r2,d4/r3,d4/r5,u5/r2,u5/r4,u5/r5,d5/r3,d5/r4,d5/r5}
    \draw (\from) -- (\to);

\end{tikzpicture}
\end{center}
\caption{The figure shows a 3-connected graph $G$ in which removing any three light vertices disconnects $G$. For each triplet $u,v,w$ of light vertices in $G$, there exists a dark vertex whose neighbors are exactly $u,v,w$.}
\label{fig:counterexample}
\end{figure}
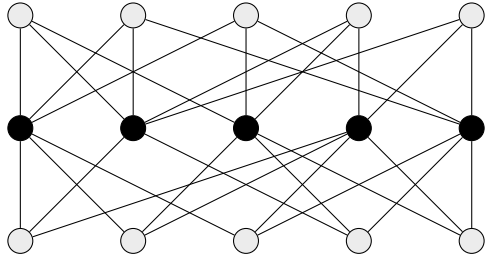

\section{Conclusion}

In Section \ref{s:maxdeg}, we provide a characterization of flexibility in terms of the maximum degree of a graph.
Moreover, we prove a more general theorem (Theorem~\ref{thm:weaker_lists_new}), which somewhat resembles a famous theorem of Erd\H{o}s, Rubin, and Taylor \cite{ErdosDlist}.
It would be very interesting to discover whether a complete characterization of flexible degree choosability exists and, moreover, how closely such a characterization would resemble Erd\H{o}s, Rubin, and Taylor's characterization of degree choosability.
Perhaps, some more structural insight might be needed in order to find such a characterization, as hinted by Figure~\ref{fig:weak_diamong}.

As we present tight bounds for list sizes needed for flexibility in graphs of bounded treedepth and graphs of treewidth 2, a natural question arises:
Is it possible to show similar bounds for graphs of bounded pathwidth?
More specifically, one can focus on the even more restricted class of (unit) interval graphs with bounded clique size.
However, even for such a restricted graph class it seems to be challenging to show a tight bound on the list size needed for flexibility.
In particular, the $k$-path seems to be a challenging example.

Another interesting direction could be a systematic search for graphs that are not flexibly choosable.
In this paper we give quite simple examples (Theorem~\ref{thm:maxdeg2} and Figure~\ref{fig:weak_diamong})
of graphs that are not flexible for the stronger reason that they do not allow a precoloring extension. These examples are not surprising, as many examples exist of graphs that do not allow a precoloring extension~\cite{Tuza1,Tuza3, Marx06, Tuza2002}. It would be interesting to find some constructions that prohibit flexibility while allowing precoloring extension. 

We finally propose an independent study of game connectivity (Definition~\ref{def:game_connectivity}), which might have interesting properties in other contexts.
To the best of our knowledge, an equivalent definition of this concept has not been studied yet in the literature.

\bigskip
\noindent\textbf{Acknowledgements.}~We thank Bojan Mohar for suggesting that we further improve our original bound of $\frac{1}{2\Delta^3}$ in Theorem~\ref{thm:maxdeg3}. 

\raggedright
\bibliographystyle{plainurl}
\bibliography{bib}

\end{document}